\documentclass{article}

\usepackage{amsmath}
\usepackage{mathrsfs}    
\usepackage{amssymb}
\usepackage{amsthm}
\usepackage{graphicx,color}
\usepackage{texdraw}
\usepackage{hyperref}
\newtheoremstyle{plain2}{\topsep}{\topsep}%
     {\itshape}
     {}
     {\bfseries}
     {.}
     {.5em}
     {\thmnumber{(#2)}\thmname{ #1}\thmnote{ #3}}

\newtheorem{teo}{Theorem}[section]
\newtheorem{prop}[teo]{Proposition}
\newtheorem{defin}[teo]{Definition}
\newtheorem{coro}[teo]{Corollary}
\newtheorem{lemma}[teo]{Lemma}

\newtheorem*{key}{Keywords}

\newtheoremstyle{definition2}{\topsep}{\topsep}%
     {}
     {}
     {\bfseries}
     {.}
     {.5em}
     {\thmnumber{(#2)}\thmname{ #1}\thmnote{ #3}}

\theoremstyle{definition}
\newtheorem{rem}[teo]{Remark}

\def\N{\mathbb{N}}

\def\R{\mathbb{R}}

\def\a{\alpha}

\def\Deh{\mathscr{D}}

\def\Ceh{\mathcal{C}}
\def\Mah{\mathcal{M}}

\def\Seh{\mathcal{S}}

\def\Feh{\mathcal{F}}
\def\Geh{\mathcal{G}}

\def\Hh{\mathscr{H}}
\def\Ss{\mathscr{S}}

\def\Bger{\mathfrak{B} }
\def\Lger{\mathfrak{L}}

\title{\sc Regularity of the nodal set of segregated critical configurations under a weak reflection law
}
\author{Hugo Tavares and Susanna Terracini}

\begin{document}
\maketitle

\begin{abstract}
We deal with a class of  Lipschitz  vector functions $U=(u_1,\ldots,u_h)$ whose components are non negative, disjointly supported and verify an elliptic equation on each support. Under a weak formulation of a reflection law, related to the Poho\u zaev identity, we prove that the nodal set is a collection of  $C^{1,\alpha}$ hyper-surfaces (for every $0<\alpha<1$), up to a residual set with small Hausdorff dimension. This result applies
 to the asymptotic limits of reaction-diffusion systems with strong competition interactions, to optimal partition problems involving eigenvalues, as well as to segregated standing waves for Bose-Einstein condensates in multiple  hyperfine spin states.
\end{abstract}

\begin{key}
Elliptic Systems, Free Boundary Problems, Monotonicity Formulae, Reflection Principle.
\end{key}


\section{Introduction}

\subsection{Statement of the results}

Let $\Omega$ be an open bounded subset of $\R^N$, with $N\geq 2$.
Our main interest is the study of the regularity of the nodal set $\Gamma_U=\{x\in \Omega:\ U(x)=0\}$ of segregated configurations $U=(u_1,\ldots,u_h)\in (H^1(\Omega))^h$ associated with systems of semilinear elliptic equations. The main result of this paper is the following.

\begin{teo}\label{teo:main_result}

Let  $U=(u_1,\ldots,u_h)\in (H^1(\Omega))^h$ be a vector of non negative  Lipschitz functions in $\Omega$,  having mutually disjoint supports:  $u_i\cdot u_j\equiv 0$ in $\Omega$ for $i\neq j$. Assume that  $U\not\equiv 0$ and
   \begin{equation*}-\Delta u_i=f_i(x,u_i) \qquad \qquad {\rm whenever }\;u_i >0\;, \ \ i=1,\ldots,h,\end{equation*}
where $f_i:\Omega\times\R^+\rightarrow \R$ are $C^1$ functions such that $f_i(x,s)=O(s)$ when $s\rightarrow 0$, uniformly in $x$. Moreover, defining for every $x_0\in \Omega$ and $r\in (0,\text{dist}(x_0,\partial \Omega))$ the energy
          \begin{equation}\nonumber
          \tilde E(r)=\tilde E(x_0,U,r)=\frac{1}{r^{N-2}}\int_{B_r(x_0)}|\nabla U|^2\;,
          \end{equation}
assume that  $\tilde E(x_0,U,\cdot)$ is an absolutely continuous function of $r$ and that it satisfies the following differential equation
   \begin{equation*}
   \frac{d}{dr}\tilde E(x_0,U,r)= \frac{2}{r^{N-2}}\int_{\partial B_r(x_0)}  (\partial_\nu U )^2\, d\sigma +\frac{2}{r^{N-1}}\int_{B_r(x_0)} \sum_{i}f_i(x,u_i)\langle \nabla u_i, x-x_0\rangle.
   \end{equation*}

Let us consider the nodal set $\Gamma_U=\{x\in \Omega:\ U(x)=0\}$. Then we have\footnote{Here, $\Hh_\text{dim}(\cdot)$ denotes the Hausdorff dimension of a set.} $\Hh_\text{dim}(\Gamma_U)\leq N-1$. Moreover there exists a set $\Sigma_U\subseteq \Gamma_U$, relatively open in $\Gamma_U$, such that
\begin{itemize}
\item $\Hh_\text{dim} (\Gamma_U\setminus \Sigma_U)\leq N-2$, and if $N=2$ then actually $\Gamma_U\setminus \Sigma_U$ is a locally finite set;
\item $\Sigma_U$ is a collection of  hyper-surfaces of class $C^{1,\alpha}$ (for every $0<\alpha<1$). Furthermore for every $x_0\in \Sigma_U$
\begin{equation}\label{eq:reflection_principle}\lim_{x\rightarrow x_0^+} |\nabla U(x)|=  \lim_{x\rightarrow x_0^-} |\nabla U(x)|\neq 0,\end{equation}
\end{itemize}
where the limits as $x\to x_0^\pm$ are taken from the opposite sides of the hyper-surface. Furthermore,  if $N=2$ then $\Sigma_U$ consists in a locally finite collection of curves meeting with equal angles at singular points.
\end{teo}

The regularity of the nodal set can be extended  up to the boundary under appropriate assumptions (see Remark \ref {rem:global_Lipschitz}).  To proceed, it is convenient to group the vector functions satisfying the assumptions of Theorem \ref{teo:main_result} in the following class.

\begin{defin}\label{defin:class_G}
We define  the class $\Geh(\Omega)$ as the set of functions $U=(u_1,\ldots,u_h)\in (H^1(\Omega))^h$, whose components are all non negative and Lipschitz continuous in the interior of $\Omega$, and such that $u_i\cdot u_j\equiv 0$ in $\Omega$ for $i\neq j$. Moreover, $U\not\equiv 0$ and it solves a system of the type
   \begin{equation}\label{eq:system_u_i}-\Delta u_i=f_i(x,u_i)-\mu_i \qquad \qquad {\rm in } \ \Deh'(\Omega)=(C^{\infty}_\text{c}(\Omega))',\ \ i=1,\ldots,h, \end{equation}
where
\begin{itemize}
     \item[(G1)] $f_i:\Omega\times\R^+\rightarrow \R$ are $C^1$ functions such that $f_i(x,s)=O(s)$ when $s\rightarrow 0$, uniformly in $x$;
\end{itemize}
\begin{itemize}
     \item[(G2)] $\mu_i\in \Mah(\Omega)=(C_0(\Omega))'$ are some nonnegative Radon measures, each supported on the nodal set $\Gamma_U=\{x\in \Omega:\ U(x)=0\}$,   \end{itemize}
and moreover
   \begin{itemize}
     \item[(G3)]  associated to system \eqref{eq:system_u_i}, if we define for every $x_0\in \Omega$ and $r\in (0,\text{dist}(x_0,\partial \Omega))$ the quantity
          \begin{equation}\nonumber
          \tilde E(r)=\tilde E(x_0,U,r)=\frac{1}{r^{N-2}}\int_{B_r(x_0)}|\nabla U|^2,
          \end{equation}
 then $\tilde E(x_0,U,\cdot)$ is an absolutely continuous function of $r$ and
   \begin{equation}\label{eq:tilde_E'_expression}
   \frac{d}{dr}\tilde E(x_0,U,r)= \frac{2}{r^{N-2}}\int_{\partial B_r(x_0)}  (\partial_\nu U )^2\, d\sigma +\frac{2}{r^{N-1}}\int_{B_r(x_0)} \sum_{i}f_i(x,u_i)\langle \nabla u_i, x-x_0\rangle.
   \end{equation}
\end{itemize}
\end{defin}
To check the equivalence between the two sets of assumptions, we observe that equation \eqref{eq:system_u_i} together with (G2) yield that $-\Delta u_i=f_i(x,u_i)$ over the set $\{u_i>0\}$. Reciprocally, if such equation holds in $\{u_i>0\}$ then \eqref{eq:system_u_i} holds in the whole $\Omega$ for a measure $\mu_i$ concentrated in $\Gamma_U$ (a proof of this fact will be provided in Lemma \ref{lemma:ext_by_0} in a similar situation). We will work from now on with this second formulation of the assumptions.

\noindent \textbf{Notations.} For any vector function $U=(u_1,\ldots,u_h)$ we define $\nabla U=(\nabla u_1,\ldots, \nabla u_h)$, $|\nabla U|^2=|\nabla u_1|^2+\ldots+|\nabla u_h|^2$, $(\partial_\nu U)^2=(\partial_\nu u_1)^2+\ldots+(\partial_\nu u_h)^2$ and $U^2=u_1^2+\ldots+u_h^2$. Moreover,
$F(x,U)=(f_1(x,u_1),\ldots,f_h(x,u_h))$. We will denote by $\{U>0\}$ the set $\{x\in \Omega:\ u_i(x)>0 \text{ for some } i \}$. The usual scalar product in $\R^N$ will be denoted by $\langle\cdot, \cdot \rangle$. Hence, with these notations, $\langle F(x,U),U\rangle=\sum_i f_i(x,u_i)u_i$ and $\langle U,\partial_\nu U \rangle= \sum_i u_i (\partial_\nu u_i)$ for instance.

\begin{rem}
(a) It is easily checked that  equation \eqref{eq:tilde_E'_expression} always holds for balls lying entirely inside one of the component supports,  as a consequence of the elliptic equation \eqref{eq:system_u_i} (see also \S\ref{subsect:heuristic}). Hence, for our class systems, (G3) represents the only interaction between the different components $u_i$  through the common boundary of their supports; as we are going to discuss in \S\ref{subsect:heuristic} this can be seen as a weak form of a reflection property through the interfaces.  Although this hypothesis may look weird and may seem hard to check in applications, it has the main advantage to  occur naturally in many situations where the vector $U$ appears as a limit configuration in problems of spatial segregation. It has to be noted indeed that a form of \eqref{eq:tilde_E'_expression} always holds for solutions of systems of interacting semilinear equations and that it persists under strong $H^1$ limits (see \S \ref{sec:applications}).

(b) Theorem \ref{teo:main_result} applies to the nodal components of solutions to a single semilinear elliptic equation of the form $-\Delta u=f(u)$. Hence, in a sense, our work generalizes \cite{Lin,HHHN:1999}. In the paper \cite{CL}, Caffarelli and Lin proved that the same conclusion of Theorem \ref{teo:main_result} holds for vector functions $U$ minimizing Lagrangian functional associated with the system. They also proved that equation \eqref{eq:tilde_E'_expression} holds for such energy minimizing configurations. On the other hand, at the end of this paper we show that \eqref{eq:tilde_E'_expression} is fullfilled also for strong limits to competition--diffusion systems, both those possessing a variational structure and those with Lotka-Volterra type interactions (see  \S \ref{sec:applications} for some applications of Theorem \ref{teo:main_result}). Inspired by our recent work \cite{uniform_holder} written in collaboration with Noris and Verzini, we found that property (G3) is a suitable substitute for the minimization property.

(c) Our theorem extends also to sign changing, complex and vector valued functions $u_i$. For the sake of simplicity we shall expose here the proof for non negative real components, highlighting in Remark \ref{rem:vector_case} the modifications needed to cover the general case.  

(d) Finally we observe that the conclusions of Theorem \ref{teo:main_result} are all of local type. Hence, the conclusion are still valid in the case $\Omega$ unbounded by applying our main theorem to each bounded subset $\Omega'\subset\Omega$.

\end{rem}

The approach here differs from the viscosity one proposed by Caffarelli in \cite{C} (which we think does not apply to elements in $\Geh(\Omega)$) and follows rather the mainstream of \cite{CL,Lin}, based upon a classical dimension reduction principle by Federer.  It has the main advantage of avoiding the {\it a priori} assumption of non degeneracy of the free boundary (which is considered for instance in \cite[\S 4]{AC}): in contrast, non degeneracy will be turn out to hold true on the non singular part of the nodal set as a consequence of the weak reflection principle.  Compared with  \cite{CL},  a major difficulty here arises from the fact that we lack the essential  information of the minimality of the solution. The techniques we present here are not mere generalizations of the ones used in \cite{CL}: we will use a different approach when proving compactness of the blowup sequences as well as when classifying the conic functions (blowup limits); finally we will exploit an inductive argument on the dimension. This will allow us to extend the results of \cite{CL} concerning the asymptotic limits of solutions of systems arising in Bose-Einstein condensation ({\it cf.} \S \ref{subsec:application_BEC}) to the case of excited state solutions.


\subsection{Motivations and heuristic considerations}\label{subsect:heuristic}

 In $\R^2$ the functions  of the form $r^{m/2}cos(m\theta/2)$ (in polar coordinates) for any integer $m\geq 2$ are good prototypes of elements in $\Geh$. The nodal sets of such functions can be divided in two parts: the regular part is a union of curves where a reflection principle holds (the absolute value of the gradient is the same when we approach each curve from opposite sides); the remaining part has small Hausdorff measure (it is a single point). Our aim is to show that this is a general fact, in any space dimension.

 More generally, let $u$ be a locally Lipschitz $H^1$--solution of $-\Delta u=f(x,u)$ in $\Omega$ for $f\in C^1(\Omega\times \R\setminus\{0\})$ with $f(x,s)=O(s)$ as $s\rightarrow 0$, uniformly in $x$. For
 $$
 \tilde E(r)=\frac{1}{r^{N-2}}\int_{B_r(x_0)}|\nabla u|^2 \, dx
 $$
it holds
\begin{equation}\label{eq:derivative_tilde_E}
\tilde E'(r)=\frac{2-N}{r^{N-1}}\int_{B_r(x_0)}|\nabla u|^2+\frac{1}{r^{N-2}}\int_{\partial B_r(x_0)}|\nabla u|^2\, d \sigma.
\end{equation}
If we integrate the following Poh\u ozaev--type (Rellich) identity in $B_r(x_0)$
\begin{equation}\label{eq:Rellich}
\text{div}\left( (x-x_0) |\nabla u|^2 - 2\langle x-x_0,\nabla u\rangle \nabla u  \right) =(N-2) |\nabla u|^2-2\langle x-x_0,\nabla u\rangle \Delta u
\end{equation}
then we obtain
$$
r\int_{\partial B_r(x_0)}|\nabla u|^2\, d\sigma = 2r \int_{\partial B_r(x_0)}(\partial_\nu u)^2 \, d\sigma + (N-2)\int_{B_r(x_0)}|\nabla u|^2+\int_{B_r(x_0)}2f(x,u) \langle \nabla u, x-x_0\rangle.
$$
This, together with \eqref{eq:derivative_tilde_E}, readily implies \eqref{eq:tilde_E'_expression} for $U=(u)$. Hence, if we define $u_1=u^+$ and $u_2=u^-$ we deduce that $(u_1, u_2)\in \Geh(\Omega)$.

In order to better motivate property (G3) and to better understand the information that it contains about the interaction between the different components $u_i$, let us show what happens in the presence of exactly two components, each satisfying an equation on its support. Suppose $h=2$ and take $U=(u_1,u_2)\in \Geh(\Omega)$ such that $\Omega\cap \partial \{u_1>0\}=\Omega\cap \partial \{u_2>0\}=\Gamma_U$. Assume sufficient regularity in order to perform the following computations (see the proof of Lemma \ref{lemma:reflection_principle} and Subsection \ref{subsec:class_S} for related discussions).
For every point $x_0$ and radius $r>0$, take identity \eqref{eq:Rellich} with $u=u_i$ ($i=1,2$) and integrate it in $\{u_i>0\}\cap B_r(x_0)$. We obtain
\begin{multline*}
r \int_{\partial B_r(x_0)\cap \{u_i>0\}} |\nabla u_i|^2\,d\sigma = 2 r \int_{\partial B_r(x_0)\cap \{u_i>0\}}\left(\partial_\nu u_i\right)^2\,d\sigma+ (N-2)\int_{B_r(x_0)\cap \{u_i>0\}}|\nabla u_i|^2 +\\
 + 2 \int_{B_r(x_0)\cap \{u_i>0\}}f_i(x,u_i)\langle \nabla u_i,x-x_0\rangle + \int_{B_r(x_0)\cap \partial \{u_i>0\}}|\nabla u_i|^2\langle x-x_0,\nu\rangle \, d\sigma.
\end{multline*}
This implies, by summing the equalities for $i=1,2$ and dividing the result by $r^{N-1}$,
\begin{multline*}
\frac{1}{r^{N-2}}\int_{\partial B_r(x_0)}|\nabla U|^2\, d\sigma =\frac{2}{r^{N-2}} \int_{\partial B_r(x_0)}\left( \partial_\nu U\right)^2 + \frac{N-2}{r^{N-1}}\int_{B_r(x_0)}|\nabla U|^2+  \\
+ \frac{2}{r^{N-1}}\int_{B_r(x_0)}\sum_{i=1}^{2} f_i(x,u_i) \langle \nabla u_i,x-x_0\rangle + \frac{1}{r^{N-1}}\int_{B_r(x_0)\cap \partial\{u_1>0\}}|\nabla u_1|^2 \langle x-x_0,\nu \rangle\, d\sigma +\\
+ \frac{1}{r^{N-1}}\int_{B_r(x_0)\cap \partial\{u_2>0\}}|\nabla u_2|^2 \langle x-x_0,\nu \rangle\,d\sigma
\end{multline*}
and
\begin{multline}\label{eq:derivative_of_E_heuristics}
\tilde E'(r) = \frac{2}{r^{N-2}} \int_{\partial B_r(x_0)}\left( \partial_\nu U\right)^2 + \frac{2}{r^{N-1}}\int_{B_r(x_0)}\sum_{i=1}^2 f_i(x,u_i) \langle \nabla u_i,x-x_0\rangle + \\
+ \frac{1}{r^{N-1}}\int_{B_r(x_0)\cap \partial\{u_1>0\}}|\nabla u_1|^2 \langle x-x_0,\nu \rangle\, d\sigma + \frac{1}{r^{N-1}}\int_{B_r(x_0)\cap \partial\{u_2>0\}}|\nabla u_2|^2 \langle x-x_0,\nu \rangle\,d\sigma.
\end{multline}
for every point $x_0$ and radius $r>0$.
Hence in this case (G3) holds if and only if the sum of the last two integrals in \eqref{eq:derivative_of_E_heuristics} is zero for every $x_0,r$, that is, $|\nabla u_1|=|\nabla u_2|$ on $\Gamma_U$. Thus, in some sense, (G3) is a weak formulation of a reflection principle.

This paper is organized as follows. In the next section we prove that elements in $\Geh(\Omega)$ satisfy a modified version of the Almgren's Monotonicity Formula; by exploiting this fact, in Section \ref{sec:blow_up_sequences} we prove convergence of blowup sequences as well as some closure properties of the class $\Geh(\Omega)$. In Section \ref{sec:hausdorff_dimension} we use the Federer's Reduction Principle in order to prove some Hausdorff estimates for the nodal sets, define the set $\Sigma_U$ (recall Theorem \ref{teo:main_result}) and prove part of Theorem \ref{teo:main_result} in dimension $N=2$. In Section \ref{sec:regularity_of_regular_free_boundary} we prove that, under an appropriate assumption, $\Sigma_U$ is an hyper-surface satisfying the reflection principle \eqref{eq:reflection_principle} and in Section \ref{sec:iteration_argument} we prove by induction in the dimension $N$ that such assumption is satisfied for every $N\geq 2$. In Section \ref{sec:riemannian} we examine the case of systems of equations on Riemannian manifolds and of operators with variable coefficients, also discussing the regularity up to the boundary. Finally in Section \ref{sec:applications} we present some applications of our theory and solve two different problems by showing that its solutions belong to the class $\Geh(\Omega)$.

\section{Preliminaries}\label{sec:preliminaries}

The functions belonging to $\Geh(\Omega)$ have a very rich structure, mainly due to property (G3), which will enable us to prove the validity of the Almgren's Monotonicity Formula (Theorem \ref{teo:Almgren's_Monotonicity_Formula} below).  With this purpose, it is more convenient to use a slightly modified version of (G3), including in the definition of the energy also a potential term. The two versions are clearly equivalent, and we will use this second formulation from now on:

\begin{itemize}
     \item[(G3)]  Define for every $x_0\in \Omega$ and $r\in (0,\text{dist}(x_0,\partial \Omega))$ the quantity
          \begin{equation}\nonumber
          E(r)= E(x_0,U,r)=\frac{1}{r^{N-2}}\int_{B_r(x_0)}\left(|\nabla U|^2-  \langle F(x,U), U  \rangle\right)
          \end{equation}
then $E(x_0,U,\cdot)$ is an absolutely continuous function on $r$ and
   \begin{equation}\label{lemma_E'_expression}
    \frac{d}{dr}E(x_0,U,r)= \frac{2}{r^{N-2}}\int_{\partial B_r(x_0)}  (\partial_\nu U )^2\, d\sigma +R(x_0,U,r),
   \end{equation}
with
   \begin{multline}
  R(x_0,U,r)= \frac{2}{r^{N-1}}\int_{B_r(x_0)} \sum_i f_i(x,u_i)\langle \nabla u_i,x-x_0 \rangle+\\
  +\frac{1}{r^{N-1}}\int_{B_r(x_0)}  (N-2)  \langle F(x,U), U\rangle - \frac{1}{r^{N-2}}\int_{\partial B_r(x_0)}  \langle F(x,U), U \rangle \,d\sigma .
   \end{multline}
\end{itemize}

\begin{rem}The definition of $E(x_0,U,r)$ (and the one of $R(x_0,U,r)$) is to be used with some caution. In fact, this quantity also depends on the function $F$ that is associated (through system \eqref{eq:system_u_i}) to each $U\in \Geh(\Omega)$. Although this function is not uniquely determined for any given $U$, we prefer to omit its reference in the definition of $E$, with some abuse of notations.\end{rem}

Furthermore define for every $x_0\in \Omega$ and $r\in(0,\text{dist}(x_0,\partial \Omega))$ the average
$$H(r)=H(x_0,U,r)=\frac{1}{r^{N-1}}\int_{\partial B_r(x_0)}U^2\, d\sigma$$
and, whenever $H(r)\neq 0$, the generalized Almgren's quotient by
$$N(r)=N(x_0,U,r)=\frac{E(x_0,U,r)}{H(x_0,U,r)}.$$

\begin{teo}\label{teo:Almgren's_Monotonicity_Formula} Given $U\in \Geh(\Omega)$ and $\tilde \Omega\Subset\Omega$, there exist \footnote{With $\displaystyle   d=\max_i \mathop{\sup_{0<s\leq \|U\|_{L^\infty(\Omega)}}}_{x\in \Omega}  |f_i(x,s)/s| $}
 $\tilde C=\tilde C(d,N,\tilde \Omega)>0$ and $\tilde r=\tilde r(d,N,\tilde \Omega)>0$ such that for every $x_0\in \tilde \Omega$ and $r\in(0,\tilde r]$ we have $H(x_0,U,r)\neq 0$, $N(x_0,U,\cdot)$ is an absolutely continuous function and
\begin{equation}\label{eq:derivative_of_N}
\frac{d}{dr}N(x_0,U,r)\geq-\tilde C (N(x_0,U,r)+1).
\end{equation}
In particular $e^{\tilde Cr}(N(x_0,U,r)+1)$ is a non decreasing function for $r\in (0,\tilde r]$ and the limit $N(x_0,U,0^+):=\lim_{r\rightarrow 0^+} N(x_0,U,r)$ exists and is finite.
Moreover,
\begin{equation}\label{eq:derivative_of_log_H}
\displaystyle \frac{d}{dr}\log(H(x_0,U,r))=\frac{2}{r}N(x_0,U,r).
\end{equation}
\end{teo}
\begin{proof} The proof follows very closely the one of Proposition 4.3 in \cite{uniform_holder}. For this reason we only present a sketch of it, stressing however the dependence of $\tilde C, \tilde r$ on $d$. Fix $U\in\Geh(\Omega)$ and take $\tilde \Omega\Subset \Omega$.  Since $U\not\equiv 0$ in $\Omega$, we can suppose without loss of generality that $U\not\equiv 0$ in $\tilde \Omega$ .

Observe that since $\Omega$ is bounded and $U$ is Lipschitz continuous in $\Omega$, $\|U\|_{L^\infty(\Omega)}<+\infty$. Hence property (G1) provides the upper bound $|f_i(x,u_i)|\leq d u_i$ for all $x\in \Omega$ and $i=1,\ldots,h$, and therefore there exists $C=C(d,N,\tilde \Omega)$ such that for every $x_0\in \tilde \Omega$ and $0<r<\text{dist}(\tilde \Omega,\partial \Omega)$,
\begin{eqnarray*}
|R(x_0,U,r)| &\leq& \frac{2d}{r^{N-1}}\int_{B_r(x_0)} \sum_i u_i |\nabla u_i| r + \frac{(N-2) d}{r^{N-1}}\int_{B_r(x_0)}U^2 + \frac{d}{r^{N-2}}\int_{\partial B_r(x_0)}U^2\, d\sigma\\
&\leq & C\left( \frac{1}{r^{N-2}}\int_{B_r(x_0)}|\nabla U|^2  + \frac{1}{r^N}\int_{B_r(x_0)}U^2 + \frac{1}{r^{N-1}}\int_{\partial B_r(x_0)}U^2\, d\sigma \right).
\end{eqnarray*}
Moreover, we have
\begin{equation}\label{eq:sum_and_subtract}
\frac{1}{r^{N-2}}\int_{B_r(x_0)}|\nabla U|^2 \leq  E(x_0,U,r) +\frac{1}{r^{N-2}}\int_{B_r(x_0)}\langle F(U),U\rangle\leq 
E(x_0,U,r) +\frac{d r^2}{r^N}\int_{B_r(x_0)}U^2
\end{equation}
and, by using Poincar\'e's inequality,
\begin{eqnarray*}
\frac{1}{r^N}\int_{B_r(x_0)}U^2 &\leq & \frac{1}{N-1}\left( \frac{1}{r^{N-2}}\int_{B_r(x_0)}|\nabla U|^2+\frac{1}{r^{N-1}}\int_{\partial B_r(x_0)} U^2 \, d\sigma  \right)\\
&\leq & \frac{1}{N-1}\left( E(x_0,U,r)+H(x_0,U,r)  \right) + \frac{r^2 C'}{r^N}\int_{B_r(x_0)} U^2.
\end{eqnarray*}
Thus we obtain the existence of  $\bar r<\text{dist}(\tilde \Omega,\partial \Omega)$ such that
\begin{equation}\label{eq:maj_of_|nabla u|^2_leq_E+H}
\frac{1}{r^N}\int_{B_r(x_0)} U^2 \leq 2 \left( E(x_0,U,r)+H(x_0,U,r)\right) \qquad \text{ for every }x_0\in \tilde \Omega,\ 0<r<\bar r,
\end{equation}
which, together with \eqref{eq:sum_and_subtract}, yields $|R(x_0,U,r)|\leq \tilde C \left(E(x_0,U,r)+H(x_0,U,r)\right)$ for some $\tilde C=\tilde C(d,N,\tilde \Omega)>0$ and for every $x_0\in\tilde \Omega,\ 0<r<\bar r$.
The function $r\mapsto H(x_0,U,r)$ is absolutely continuous and for almost every $r>0$
$$
\frac{d}{dr}H(x_0,U,r)=\frac{2}{r^{N-1}}\int_{\partial B_r(x_0)}\langle U, \partial_\nu U\rangle \,d\sigma
$$
(to check it, use a sequence of smooth functions approximating $U$).
Moreover if we multiply system \eqref{eq:system_u_i} by $U$, integrate by parts in $B_r(x_0)$ and take into account property (G2) we can rewrite $E$ as
$$
E(x_0,U,r)=\frac{1}{r^{N-2}}\int_{\partial B_r(x_0)}\langle U , \partial_\nu U\rangle \,d\sigma.
$$
Thus, by performing a direct computation, identity  \eqref{eq:derivative_of_log_H} holds whenever $H(x_0,U,r)>0$ for $r<\bar r$, as well as
$$\frac{d}{dr}N(x_0,U,r)\geq \frac{R(x_0,U,r)}{H(x_0,U,r)}\geq -\tilde C \frac{E(x_0,U,r)+H(x_0,U,r)}{H(x_0,U,r)},$$which provides \eqref{eq:derivative_of_N}.

The only thing left to prove is that $H(x_0,U,r)>0$ for every $x\in \tilde \Omega$ and small $r>0$. Now, since $H(x_0,U,\cdot)$ solves the equation $H'(r)=a(r)H(r)$ with $a(r)=N(r)/r$, one can prove that $\Gamma_U$ has an empty interior. Take $\tilde r< \bar r$ such that
\begin{equation}\label{eq:Almgren_acessory_result}
-\Delta u_i\leq f_i(x,u_i)\leq d u_i\leq \lambda_1(B_{\tilde r})u_i
\end{equation}
for all $i$ (where $\lambda_1$ denotes the first eigenvalue of $-\Delta$ in $H^1_0$).  If there were $x_0\in \tilde \Omega$ and $0<r<\tilde r$ such that $H(x_0,U,r)=0$, then by multiplying inequality \eqref{eq:Almgren_acessory_result} by $u_i$ and integrating by parts in $B_r(x_0)$ we would obtain $U\equiv 0$ in $B_r(x_0)$, a contradiction.
Hence $H(x_0,U,r)>0$ whenever $x_0\in \tilde \Omega$, $0<r<\tilde r$.
 \end{proof}

\begin{rem}
At this point we would like to stress that the hypotheses in $\Geh(\Omega)$ can be weakened. In \cite[Proposition 4.1]{uniform_holder}, by making use of the previous Almgren's Monotonicity Formula, it is shown that if in $\Geh(\Omega)$ we replace the Lipschitz continuity assumption with $\alpha$--H\"older continuity for every $\alpha\in(0,1)$, then actually each element $U\in \Geh(\Omega)$ is Lipschitz continuous. For other general considerations, see also Remark \ref{rem:global_Lipschitz} .
\end{rem}

\begin{rem}\label{rem:N_is_increasing}
If $U\in\Geh(\Omega)$ has as associated function $F\equiv 0$, then $R(x_0,U,r)\equiv 0$ and by repeating the previous procedure we conclude that in this case $N(x_0,U,r)$ is actually a non decreasing function.
\end{rem}

\begin{rem}\label{rem:Gamma_empty_interior}
As observed in the above proof, $\Gamma_U$ has an empty interior whenever $U\in \Geh(\Omega)$.
\end{rem}

Another simple consequence of the monotonicity result is the following comparison property (which, with $r_2=2 r_1$, is the so called doubling property).
\begin{coro}\label{coro:doubling}
Given $U\in \Geh(\Omega)$ and $\tilde \Omega\Subset \Omega$, there exist $\tilde C>0$ and $\tilde r>0$ such that
$$H(x_0,U,r_2)\leq H(x_0,U,r_1)\left(\frac{r_2}{r_1}\right)^{2\tilde C}$$ for every $x_0\in \tilde \Omega$, $0<r_1<r_2\leq \tilde r$.
\end{coro}

\begin{proof}
For each $U$ and $\tilde\Omega$ fixed, let $\tilde C$ and $\tilde r$ be the associated constants according to the previous theorem. Let also $\displaystyle C:=\sup_{x_0\in \tilde \Omega}|N(x_0,U,\tilde r)|<\infty$. Then
\begin{eqnarray*}
\frac{d}{dr}\log\left(H(x_0,U,r)\right) &=& \frac{2}{r}N(x_0,U,r) = \frac{2}{r}\left((N(x_0,U,r)+1)e^{\tilde C r}e^{-\tilde C r}-1   \right)\\
&\leq& \frac{2}{r}\left((N(x_0,U,\tilde r)+1)e^{\tilde C \tilde r}e^{-\tilde C r}-1   \right)\\
&\leq& \frac{2}{r}\left((C+1)e^{\tilde C \tilde r}-1\right)=: \frac{2\bar C}{r}
\end{eqnarray*}
for every $0<r<\tilde r$. Now we integrate between $r_1$ and $r_2$, $0<r_1<r_2\leq \tilde r$,
obtaining
$$\frac{H(x_0,U,r_2)}{H(x_0,U,r_1)}\leq \left(\frac{r_2}{r_1}\right)^{2\bar C},$$ as desired.
\end{proof}

\begin{coro}\label{coro:N_geq_1}
For any $U\in \Geh(\Omega)$ and $x_0\in \Gamma_U$, we have $N(x_0,U,0^+)\geq 1$.
\end{coro}
\begin{proof}
Suppose not. Since the limit $N(x_0,U,0^+)$ exists, we obtain the existence of $\bar r$ and $\varepsilon$ such that $N(x_0,U,r)\leq 1-\varepsilon$ for all $0\leq r\leq \bar r$. By Theorem \ref{teo:Almgren's_Monotonicity_Formula} we have that in this interval (by possibly replacing $\bar r$ with a smaller radius)
$$\frac{d}{dr}\log(H(x_0,U,r))\leq \frac{2}{r}(1-\varepsilon).$$
Integrating this inequality between $r$ and $\bar r$ ($r<\bar r$) yields
$$\frac{H(x_0,U,\bar r)}{H(x_0,U,r)}\leq \left(\frac{\bar r}{r}\right)^{2(1-\varepsilon)}$$
which, together with the fact that $U$ is a Lipschitz continuous function at $x_0$ and that $U(x_0)=0$ implies
$$C r^{2(1-\varepsilon)}\leq H(x_0,U,r)\leq C' r^2,$$a contradiction for small $r$.
\end{proof}

\begin{coro}\label{coro:N_upper_semi_continuous}
The map $\Omega\rightarrow [1,+\infty)$, $x_0\mapsto N(x_0,U,0^+)$ is upper semi-continuous.
\end{coro}
\begin{proof}
Take a sequence $x_n\rightarrow x$ in $\Omega$. By Theorem \ref{teo:Almgren's_Monotonicity_Formula} there exists a constant $C>0$ such that for small $r>0$
$$
N(x_n,U,r)=(N(x_n,U,r)+1)e^{C r}e^{-C r}-1\geq (N(x_n,U,0^+)+1)e^{-C r}-1.
$$
By taking the limit superior in $n$ and afterwards the limit as $r\rightarrow 0^+$ we obtain $N(x,U,0^+)\geq \limsup_n N(x_n,U,0^+)$.
\end{proof}


\section{Compactness of blowup sequences}\label{sec:blow_up_sequences}

All techniques presented in this paper involve a local analysis of the solutions, which will be performed via a blowup procedure. Therefore in this section we start with the study of the behavior of the class $G(\Omega)$ under rescaling, which will be followed by a convergence result for blowup sequences. This will be a key tool in the subsequent arguments.

Fix $U\in \Geh(\Omega)$ and let $f_i, \mu_i$ be associated functions and measures (respectively) in the sense of Definition \ref{defin:class_G} ({\it i.e.}, such that \eqref{eq:system_u_i} holds). For every fixed $\rho, t>0$ and $x_0\in \Omega$ define the rescaled function
	$$V(x)=\frac{1}{\rho}U_{x_0,t}(x)=\frac{U(x_0+tx)}{\rho}, \qquad \qquad {\rm for }\ x\in \Lambda:=\frac{\Omega-x_0}{t}.$$
It is straightforward to check that $V$ solves the system
\begin{equation}\label{eq:system_rescaled}
-\Delta v_i= g_i(x, v_i)-\lambda_i, \qquad \qquad {\rm in } \ \Deh'(\Lambda),\ \ i=1,\ldots,h,
\end{equation}
where
$$g_i(x,s)= \frac{t^2}{\rho}f_i(x_0+tx,\rho s) \qquad \text{and} \qquad \lambda_i(E):=\frac{1}{\rho t^{N-2}}\mu_i(x_0+t E)\ \
\text{for every Borel set } E \text{ of } \Lambda.$$
Indeed, for any given $\varphi\in \Deh(\Lambda)$,
\begin{eqnarray*}
&\displaystyle \int_{\Lambda}\left(\nabla v_i \cdot \nabla \varphi - g_i(x,v_i)\varphi \right)+\int_{\Lambda}\varphi \,d \lambda_i  = \qquad \qquad \qquad \qquad \qquad \qquad \qquad \qquad \qquad & \\
&\displaystyle =\int_{\Lambda}\left(\frac{t}{\rho}\nabla u_i(x_0+tx)\cdot \nabla \varphi-\frac{t^2}{\rho}f_i(x_0+t x, u_i(x_0+tx))\varphi\right)\, dx+\frac{1}{\rho t^{N-2}}\int_{\Lambda} \varphi(x)\,d\mu_i(x_0+t\cdot) &\\
&\displaystyle =\frac{1}{\rho t^{N-2}}\int_\Omega \left( \nabla u_i\cdot \nabla \left(\varphi((x-x_0)/t\right))-f_i(x, u_i)\varphi((x-x_0)/t)\right)\, dx+\frac{1}{\rho t^{N-2}}\int_\Omega \varphi((x-x_0)/t)\, d \mu_i(x)=0.&
\end{eqnarray*}

In this setting, for any $y_0\in \Lambda$ and $r\in(0,d(y_0,\partial \Lambda))$,
$$E(y_0,V,r)=\frac{1}{r^{N-2}}\int_{B_r(y_0)}\left(|\nabla V|^2-\langle G(x,V), V\rangle \right)$$and the following identities hold:
\begin{equation}\label{eq:almgren_identities1}
E(y_0,V,r)=\frac{1}{\rho^2}E(x_0+t y_0,U,tr),\qquad H(y_0,V,r)=\frac{1}{\rho^2}H(x_0+t y_0,U,tr),\\
\end{equation}
and hence
\begin{equation}\label{eq:almgren_identities2}
N(y_0,V,r)=N(x_0+t y_0,U,tr).
\end{equation}
Moreover,

\begin{prop}\label{prop:invariance_under_translation}
With the previous notations, $\displaystyle V\in \Geh(\Lambda).$
\end{prop}

\begin{proof}
At this point the only thing left to prove is property (G3). In order to check its validity, just observe that by using \eqref{eq:almgren_identities1}  and by performing a change of variables of the form $x=x_0+t y$, we obtain
\begin{equation*}
\begin{split}
\displaystyle \frac{d}{dr}E(y_0,V,r)&= \displaystyle \frac{d}{dr}\frac{1}{\rho^2}E(x_0+t y_0,U,tr)=\frac{t}{\rho^2}\frac{dE}{dr}(x_0+t y_0,U,tr)\\
&=\displaystyle \frac{2 t}{\rho^2 (tr)^{N-2}}\int_{\partial B_{tr}(x_0+ty_0)}(\partial_\nu U)^2\,d\sigma + \frac{t}{\rho^2}R(x_0+t y_0,U,tr)\\
&= \displaystyle \frac{2}{r^{N-2}}\int_{\partial B_r(y_0)}(\partial_\nu V)^2+\frac{t}{\rho^2}R(x_0+t y_0,U,tr),
\end{split}
\end{equation*}
and
\begin{equation*}
\begin{split}
\displaystyle \frac{t}{\rho^2}R(x_0+t y_0,U,tr)=\ & \displaystyle \frac{2t}{\rho^2 (tr)^{N-1}}\int_{B_{tr}(x_0+t y_0)} \sum_i f_i(x,u_i) \langle \nabla u_i, x-(x_0+t y_0)\rangle + \\
&+ \displaystyle \frac{t}{\rho^2 (tr)^{N-1}}\int_{B_{tr}(x_0+t y_0)} (N-2) \langle F(U), U\rangle -\frac{t}{\rho^2(tr)^{N-2}}\int_{\partial B_{tr}(x_0+t y_0)}\langle F(U), U \rangle \, d\sigma\\
=\ & \displaystyle \frac{2}{r^{N-1}} \int_{B_r(y_0)} \sum_i g_i(x,v_i)\langle \nabla v_i, x-y_0\rangle+\\
&+\displaystyle \frac{1}{r^{N-1}}\int_{B_r(y_0)}  (N-2)\langle G(x,V), V\rangle
  - \frac{1}{r^{N-2}}\int_{\partial B_r(y_0)}\langle G(x,V), V\rangle \, d\sigma  \\
=\ & R(y_0,V,r).
\end{split}
\end{equation*}

\end{proof}

Next we turn our attention  to the convergence of blowup sequences. Let $\tilde \Omega\Subset \Omega$ and take some sequences $x_k\in \tilde \Omega$, $t_k\downarrow 0$. We define a blowup sequence by
$$U_k(x)=\frac{U(x_k+t_k x)}{\rho_k},\qquad\qquad  \text{for } x\in \frac{\Omega-x_k}{t_k},$$
with
$$\rho_k^2=\|U(x_k+t_k \cdot)\|^2_{L^2(\partial B_1(0))}=\frac{1}{t_k^{N-1}}\int_{\partial B_{t_k}(x_k)}U^2\, d\sigma = H(x_k,U,t_k).$$
We observe that $\|U_k\|_{L^2(\partial B_1(0))}=1$ and, by the previous computations, $U_k\in \Geh((\Omega-x_k)/t_k)$ and
\begin{equation}\label{eq:system_u_i_rescaled}
-\Delta u_{i,k}=f_{i,k}(x,u_{i,k})-\mu_{i,k},
\end{equation}
with
$$f_{i,k}(s)=\frac{t_k^2}{\rho_k}f_{i}(x_k+t_k x, \rho_k s),\qquad \qquad \mu_{i,k}(E)=\frac{1}{\rho_k t_k^{N-2}}\mu_i(x_k+t_k E).$$
We observe moreover that $(\Omega-x_k)/t_k$ converges to $\R^N$ because $d(x_k,\partial \Omega)\geq \text{dist}(\tilde \Omega,\partial \Omega)>0$ for every $k$. In order to simplify the upcoming statements, we introduce the following auxiliary class of functions.
\begin{defin} We say that $U\in \Geh_\text{loc}(\R^N)$ if $U\in \Geh(B_R(0))$ for every $R>0$.
\end{defin}

In the remaining part of this section we will prove the following convergence result and present some of its main consequences.

\begin{teo}\label{teo:blow_up_convergence}
Under the previous notations there exists a function $\bar U\in \Geh_\text{loc}(\R^N)$ such that, up to a subsequence, $U_k\rightarrow \bar U$ in $C^{0,\alpha}_\text{loc}(\R^N)$ for every $0<\alpha<1$ and strongly in $H^1_\text{loc}(\R^N)$. More precisely there exist $\bar \mu_i\in \Mah_\text{loc}(\R^N)$, concentrated on $\Gamma_{\bar U}$, such that $\mu_{i,k}\rightharpoonup \bar \mu_i$ weak-- $\star$ in $\Mah_\text{loc}(\R^N)$, $\bar U$ solves
\begin{equation}\label{eq:limiting_equation}
-\Delta \bar u_i=-\bar \mu_i \qquad \qquad \text{in}\ \Deh'(\R^N)
\end{equation}
and it holds
\begin{equation}\label{eq:derivative_of_E_limit}
 \frac{d}{dr} E(x_0,\bar U,r)=\frac{2}{r^{N-2}}\int_{\partial B_r(x_0)}(\partial_\nu \bar U)^2\, d\sigma \qquad \qquad \text{for a.e. } r>0 \text{ and every } x_0\in \R^N,
\end{equation}
where $E(x_0,\bar U,r)=\frac{1}{r^{N-1}}\int_{B_R(x_0)}|\nabla \bar U|^2$ is the energy associated with \eqref{eq:limiting_equation}.
\end{teo}

The proof will be presented in a series of lemmata. 

\begin{lemma}\label{lemma:blow_up_convergence_1}
There exists $\tilde r>0$ such that for every $0<r<\tilde r$ and $x_0\in \tilde \Omega$ we have
$$\frac{1}{r^{N-2}}\int_{B_r(x_0)}|\nabla U|^2 + \frac{1}{r^{N-1}}\int_{\partial B_r(x_0)}U^2 \, d\sigma \leq 2(E(x_0,U,r)+H(x_0,U,r)).$$
\end{lemma}
\begin{proof}
This result is a direct consequence of inequalities \eqref{eq:sum_and_subtract} and \eqref{eq:maj_of_|nabla u|^2_leq_E+H}.
\end{proof}

\begin{lemma}\label{lemma:blow_up_convergence_2}
For any given $R>0$ we have $\|U_k\|_{H^1(B_R(0))}\leq C$, independently of $k$.
\end{lemma}

\begin{proof}
Let $\tilde C$ and $\tilde r$ be constants such that Theorem \ref{teo:Almgren's_Monotonicity_Formula}, Corollary \ref{coro:doubling} and Lemma \ref{lemma:blow_up_convergence_1} hold for the previously fixed domain $\tilde \Omega$. We have, after taking $k$ so large that $t_k, t_k R\leq \tilde r$,
\begin{eqnarray*}
\int_{\partial B_R(0)}U_k^2 \, d\sigma &=& \frac{1}{\rho_k^2}\int_{\partial B_R(0)}U^2(x_k+t_k x)\, d\sigma = \frac{1}{\rho_k^2 t_k^{N-1}}\int_{\partial B_{t_k R}(x_k)}U^2 \,d\sigma\\
&=& R^{N-1}\frac{H(x_k,U,t_k R)}{H(x_k,U,t_k)}\leq R^{N-1}\left(\frac{t_k R}{t_k}\right)^{2 \tilde C}=: C(R)R^{N-1}
\end{eqnarray*}
(by Corollary \ref{coro:doubling}).
Moreover,
\begin{eqnarray*}
\frac{1}{R^{N-2}}\int_{B_R(0)}|\nabla U_k|^2&=&  \frac{  H(0,U_k,R)}{H(0,U_k,R)}\frac{1}{R^{N-2}}\int_{B_R(0)}|\nabla U_k|^2\\
&\leq& \frac{C(R)}{H(0,U_k,R)}\left(\frac{1}{R^{N-2}}\int_{B_R(0)}|\nabla U_k|^2+\frac{1}{R^{N-1}}\int_{\partial B_R(0)}U_k^2 \, d\sigma \right)-C(R)\\
&=& \frac{C(R)}{H(x_k,U,t_k R)}\left( \frac{1}{(t_kR)^{N-2}}\int_{B_{t_kR}(x_k)}|\nabla U|^2+\frac{1}{(t_kR)^{N-1}}\int_{\partial B_{t_k R}(x_k)}U^2 \, d\sigma \right)-C(R)\\
&\leq& \frac{2C(R)}{H(x_k,U,t_k R)}\left(E(x_k,U,t_k R)+H(x_k,U,t_k R)\right)-C(R)\\
&=& 2C(R) N(x_k,U,t_kR)+C(R)\leq 2C(R)(N(x_k,U,\tilde r)+1)e^{\tilde C \tilde r}-C(R)\leq  C'(R),
\end{eqnarray*}
where we have used identities \eqref{eq:almgren_identities1}, the continuity of the function $x\mapsto N(x,U,\tilde r)$, as well as Theorem \ref{teo:Almgren's_Monotonicity_Formula} and Lemma \ref{lemma:blow_up_convergence_1}.
\end{proof}

\begin{rem}\label{rem:blow_up_convergence_3}
Since $-\Delta u_{i,k}\leq f_{i,k}(x,u_{i,k})=\frac{t_k^2}{\rho_k}f_i(x_k+t_k x,u_i(x_k+t_k x))\leq d t_k^2 u_{i,k}$ (by property (G1)), then a standard Brezis-Kato type argument together with the $H^1_\text{loc}$--boundedness provided by the previous lemma yield that $\|U_k\|_{L^\infty(B_R(0))}\leq C(R)$ for every $k$.
\end{rem}

\begin{lemma}\label{lemma:blow_up_convergence_3}
For any given $R>0$ there exists $C>0$ such that $\|\mu_{i,k}\|_{\Mah(B_R(0))}=\mu_{i,k}(B_R(0))\leq C$ for every $k\in\N$ and $i=1,\ldots,h$.
\end{lemma}
\begin{proof}
We multiply equation \eqref{eq:system_u_i_rescaled} by $\varphi$, a cut-off function such that $0\leq \varphi\leq 1$, $\varphi=1$ in $B_R(0)$ and $\varphi=0$ in $\R^N\setminus B_{2R}(0)$. It holds
\begin{eqnarray*}
\mu_{i,k}(B_R(0))&\leq& \int_{B_{2R}(0)}\varphi \, d\mu_{i,k} = -\int_{B_{2R}(0)} \nabla u_{i,k}\cdot \nabla \varphi + \int_{B_{2R}(0)}f_{i,k}(x,u_{i,k})\varphi\\
&\leq & C(R)\| \nabla u_{i,k}\|_{L^2(B_{2R}(0))}+C(R)\|u_{i,k}\|_{L^\infty(B_{2R}(0))}\leq \tilde C(R),
\end{eqnarray*}
by Lemma \ref{lemma:blow_up_convergence_2} and Remark \ref{rem:blow_up_convergence_3}.
\end{proof}

So far we have proved the existence of a non trivial function $\bar U\in H^1_\text{loc}(\R^N)\cap L^\infty_\text{loc}(\R^N)$ and $\bar \mu_i\in \Mah_\text{loc}(\R^N)$ such that (up to a subsequence)
$$U_k\rightharpoonup \bar U \qquad  \text{ in } H^1_\text{loc}(\R^N),$$
$$\mu_{i,k}\stackrel{\star}{\rightharpoonup} \bar \mu_i \qquad  \text{ in } \Mah_\text{loc}(\R^N).$$
Moreover since $-\Delta u_{i,k}=f_{i,k}(x,u_{i,k})-\mu_{i,k}$ and $\|f_{i,k}(x,u_{i,k})\|_{L^\infty(B_R(0))}\leq d t_k^2 \|u_{i,k}\|_{L^\infty(B_R(0))}\rightarrow 0$  then
$$-\Delta \bar u_i=-\bar \mu_i \qquad \text{ in } \Deh'(\R^N).$$
The next step is to prove that the convergence $U_k\rightarrow \bar U$ is indeed strong in $H^1_\text{loc}$ and in $C^{0,\alpha}_\text{loc}$ (see Lemmata \ref{lemma:blow_up_convergence_6} and \ref{lemma:blow_up_convergence_7} ahead). These facts will come out as a byproduct of some uniform Lipschitz estimates.

\begin{lemma}\label{lemma:blow_up_convergence_4}
Fix $R>0$. Then there exist constants $C, \bar r, \bar k>0$ such that for $k\geq \bar k$ we have
$$
H(x,U_k,r)\leq C r^2
$$for $0<r<\bar r$ and $x\in B_{2R}(0)\cap \Gamma_{U_k}$.
\end{lemma}
\begin{proof}
We recall that $U_k\in \Geh(B_{3R}(0))$ for $k$ large and apply Theorem \ref{teo:Almgren's_Monotonicity_Formula} to the subset $B_{2R}(0)\Subset B_{3R}(0)$. First of all observe that for $0<s\leq \| U_k \|_{L^\infty(B_{3R}(0))}$ it holds $\rho_k s\leq \|U(x_k+t_k\cdot)\|_{L^\infty(B_{3R}(0))}\leq C'(R)$ ({\it cf.} Remark \ref{rem:blow_up_convergence_3}) and hence by taking into account property (G1) we obtain the existence of $\bar k>0$ such that
  $$
   \max _i \mathop{\sup_{0<s\leq \|U_k\|_{L^\infty (B_{3R}(0))}} }_{x\in B_{3R}(0)}|f_{i,k}(x,s)/s|\leq 
   \max_i \mathop{\sup_{0<\rho_k s\leq C'(R)}}_{x\in B_{3R}(0)} t_k^2 |f_i(x_k+t_k x, \rho_k s)/(\rho_k       s)|\leq 1
   $$ 
for $k\geq \bar k$. Therefore there exist $\bar C,\bar r>0$ independent of $k$ such that the function $r\mapsto (N(x,U_k,r)+1) e^{\bar C r}$ is non decreasing for $x\in B_{2R}(0)$ and $0<r<\bar r$.
If we suppose moreover that $x\in \Gamma_{U_k}$ then Corollary \ref{coro:N_geq_1} yields
\begin{equation*}
\frac{d}{dr}\log\left( \frac{H(x,U_k,r)}{r^2} \right)=\frac{2}{r}(N(x,U_k,r)-1)=\frac{2}{r}((N(x,U_k,r)+1)e^{\bar C r} e^{-\bar C r}-2)\geq \frac{4}{r}(e^{-\bar C r}-1)
\end{equation*}
which implies (after integration)
$$\frac{H(x,U_k,r)}{r^2}\leq \frac{H(x,U_k,\bar r)}{\bar r^2} \exp\left( \int_0^{\bar r}\frac{4}{\rho}(1-e^{-\bar C \rho})\, d\rho \right)
\leq C' \|U_k\|^2_{L^\infty(B_{2R+\bar r}(0))}\leq C.$$
\end{proof}

Next we state a technical and general lemma, which proof we left to the reader (it is an easy adaptation of the standard proof of the mean value theorem for subharmonic functions, see for instance \cite[Theorem 2.1]{GT}).

\begin{lemma}\label{lemma:blow_up_convergence_5}
Let $u\in C^2(\Omega)$ satisfy $-\Delta u\leq a u$ for some $a>0$. Then for any ball $B_R(x_0)\Subset\Omega$ we have
$$u(x_0)\leq \frac{1}{|B_R|}\int_{B_R(x_0)}u+\frac{a}{2(N+2)}R^2\|u\|_{L^\infty(B_R(x_0))}.$$
\end{lemma}
Now we are in position to prove the $C^{0,1}_\text{loc}$--boundedness of $\{U_k\}$.
\begin{lemma}\label{lemma:blow_up_convergence_6}
For every $R>0$ there exists $C>0$ (independent of $k$) such that
$$\|U_k\|_{C^{0,1}(\bar B_R(0))}\leq C \qquad \qquad \text{for every }k.$$
\end{lemma}

\begin{proof}
Suppose, without loss of generality, that
$$[U_k]_{C^{0,1}(\bar B_R(0))}:=\max_{i=1,\ldots,h} \mathop{\max_{x,y\in \bar B_R(0)}}_{x\neq y} \frac{|u_{i,k}(x)-u_{i,k}(y)|}{|x-y|}=\frac{|u_{1,k}(y_k)-u_{1,k}(z_k)|}{|y_k-z_k|}.$$
Define $r_k=|y_k-z_k|$ and suppose that
$$
2 R_k:= \max\{ \text{dist}(y_k,\Gamma_{u_{1,k}}),\text{dist}(z_k,\Gamma_{u_{1,k}})\}=\text{dist}(z_k,\Gamma_{u_{1,k}}).
$$
We can assume $\text{dist}(z_k,\Gamma_{u_{1,k}})>0$, otherwise $[U_k]_{C^{0,1}}=0$ and the lemma trivially holds. Moreover, in this case we obtain that $\text{dist}(z_k,\Gamma_{u_{1,k}})=\text{dist}(z_k,\Gamma_{U_k})$ because $u_{i,k}\cdot u_{j,k}=0$ for $i\neq j$.

We divide the proof in several cases. The idea is to treat the problem according to the interaction between $y_k$, $z_k$ and $\Gamma_{U_k}$.

\noindent CASE 1. $r_k\geq \gamma$ for some $\gamma>0$.

By the $L^\infty$--boundedness of $U_k$, it holds
$$\frac{|u_{1,k}(y_k)-u_{1,k}(z_k)|}{|y_k-z_k|}\leq \frac{2\|U_k\|_{L^\infty(B_R(0))}}{\gamma}\leq C.$$

\noindent CASE 2. $r_k\rightarrow 0$ and $R_k\geq \gamma$ for some $\gamma>0$.

Observe that in $B_{R_k}(z_k)$ the function $u_{1,k}$ solves the equation $-\Delta u_{1,k}=f_{1,k}(x,u_{1,k})$. By taking  $q>N$ we obtain the existence of $C>0$ independent of $k$ such that
$$
[u_{1,k}]_{C^{0,1}(B_{\gamma/2}(z_k))}\leq C\left( \|u_{1,k}\|_{L^q(B_\gamma(z_k))}+\| f_{1,k}(x,u_{1,k})\|_{L^q(B_\gamma(z_k))}\right)\leq C \gamma^{N/q}\|u_{1,k}\|_{L^\infty(B_\gamma(z_k))}\leq C'.
$$
Since $y_k\in B_{\gamma/2}(z_k)$ for large $k$, then $[u_{1,k}]_{C^{0,1}(\bar B_R(0))}\leq C$ in this case.

\noindent CASE 3. $R_k,\ r_k\rightarrow 0$ and $R_k/r_k\leq C$.

Notice first of all that we can apply Lemma \ref{lemma:blow_up_convergence_5} to $u_{1,k}^2$ in $B_{R_k}(z_k)$, obtaining
\begin{equation*}
u_{1,k}^2(z_k)\leq \frac{1}{|B_{R_k}|}\int_{B_{R_k}(z_k)}u_{1,k}^2 +CR_k^2.
\end{equation*}
On the other hand, let $w_k\in \Gamma_{U_k}\cap B_{2R}(0)$ be such that $\text{dist}(z_k,\Gamma_{U_k})=|z_k-w_k|$. Lemma \ref{lemma:blow_up_convergence_4} then yields the existence of $C>0$ and $\bar r>0$ such that for $k$ large
$$H(w_k,U_k,r)\leq C r^{2},\qquad \text{which implies} \qquad \frac{1}{|B_r|}\int_{B_r(w_k)}U_k^2\leq C r^{2} \qquad \text{for }r\leq \bar r.$$
By taking $k$ sufficiently large in such a way that $3R_k\leq \bar r$, we have
\begin{eqnarray*}
u_{1,k}^2(z_k)&\leq & \frac{1}{|B_{R_k}|}\int_{B_{R_k}(z_k)}u_{1,k}^2 + C R_k^2\leq \frac{C}{|B_{3R_k}|}\int_{B_{3R_k}(w_k)}U_k^2 + C R_k^2\leq C' R_k^{2}\leq C'' r_k^{2}.
\end{eqnarray*}
As for $y_k$, either $\text{dist}(y_k,\Gamma_{u_{1,k}})=0$ (and $u_{1,k}(y_k)=0$) or $\text{dist}(y_k,\Gamma_{u_{1,k}})=\text{dist}(y_k,\Gamma_{U_k})>0$ and we can apply the same procedure as before (with $R_k$ replaced by $\text{dist}(y_k,\Gamma_{U_k})/2$ - observe that $\text{dist}(y_k,\Gamma_{U_k})\leq 3 R_k \rightarrow 0$), obtaining $u_{1,k}^2(y_k)\leq C \text{dist}^2(y_k,\Gamma_{U_k})\leq C' R_k^2\leq C'' r_k^2$. Hence
$$|u_{1,k}(y_k)-u_{1,k}(z_k)|^2\leq  C r_k^2= C |z_k-y_k|^{2}. $$

\noindent CASE 4. $R_k,\ r_k\rightarrow 0$ and $R_k/r_k\rightarrow +\infty$.

In this case observe that once again if we fix $q>N$ there exists $C>0$  such that
\begin{eqnarray*}
[u_{1,k}]_{C^{0,1}(\bar B_{R_k/2}(z_k))}&\leq& C R_k^{-1}\left( R_k^{-N/q}\| u_{1,k} \|_{L^q(B_{R_k}(z_k))} + R_k^{2-N/q}\|f_{1,k}(x,u_{1,k})\|_{L^q(B_{R_k}(z_k))}\right)\\
&\leq& C\left( R_k^{-1} \|u_{1,k}\|_{L^\infty (B_{R_k}(z_k))}+R_k\right).
\end{eqnarray*}
Arguing as in case 3, we prove the existence of $C>0$ such that for large $k$ and for every $x\in B_{R_k}(z_k)$ it holds $u_{1,k}^2(x)\leq C\text{dist}^2(x,\Gamma_{U_k})\leq C' R_k^2$, and thus $[u_{1,k}]_{C^{0,1}(\bar B_{R_k/2}(z_k))}\leq C$. Since $y_k\in B_{R_k/2}(z_k)$ for large $k$, the proof is complete.
\end{proof}

By the compact embeddings $C^{0,1}(B_R(0))\hookrightarrow C^{0,\alpha}(B_R(0))$ for $0<\alpha<1$ we deduce the existence of a converging subsequence  $U_k\rightarrow \bar U$ in $C^{0,\alpha}_\text{loc}$. Now we pass to the proof of the $H^1$-- strong convergence, after which we finish the proof of Theorem \ref{teo:blow_up_convergence}.

\begin{lemma}\label{lemma:blow_up_convergence_7}
For every $R>0$ we have (up to a subsequence) $U_k\rightarrow \bar U$ strongly in $H^1(B_R(0))$.
\end{lemma}
\begin{proof}
We already know that the following equations are satisfied in $\Deh'(B_{2R}(0))$ (for every $i=1,\ldots,h$):
$$-\Delta u_{i,k}=f_{i,k}(x,u_{i,k})-\mu_{i,k},\qquad \qquad -\Delta \bar u_i=-\bar \mu_i.$$
If we subtract the second equation from the first one and multiply the result by $(u_{i,k}-\bar u_i)\varphi$ (where $\varphi$ is a cut-off function such that $0\leq \varphi\leq 1$, $\varphi=1$ in $B_R(0)$ and $\varphi=0$ in $\R^N\setminus B_{2R}(0)$), we obtain
\begin{equation*}\begin{split}
\int_{B_{2R}(0)} |\nabla(u_{i,k}-\bar u_i)|^2\varphi+\int_{B_{2R}(0)}\nabla (u_{i,k}-\bar u_i)\cdot \nabla \varphi\, (u_{i,k}-\bar u_i)=\int_{B_{2R}(0)} f_{i,k}(x,u_{i,k})(u_{i,k}-\bar u_i)\varphi-\\
-\int_{B_{2R}(0)}(u_{i,k}-\bar u_i)\varphi\,d \mu_{i,k} + \int_{B_{2R}(0)}(u_{i,k}-\bar u_i)\varphi\,d \bar \mu_i . \end{split}
\end{equation*}
Now we can conclude by observing that
$$\left|\int_{B_{2R}(0)}\nabla (u_{i,k}-\bar u_i)\cdot \nabla \varphi\, (u_{i,k}-\bar u_i)\right|\leq C \|u_{i,k}-\bar u_i\|_{L^\infty(B_{2R}(0))} \| \nabla u_{i,k}\|_{L^2(B_{2R}(0))}\rightarrow 0,$$
$$\left| \int_{B_{2R}(0)} f_{i,k}(x,u_{i,k})(u_{i,k}-\bar u_i)\varphi \right|\leq C\| u_{i,k}\|_{L^\infty(B_{2R}(0))} \|u_{i,k}-\bar u_i\|_{L^\infty(B_{2R}(0))} \rightarrow 0,  $$
and
$$\left| \int_{B_{2R}(0)}-(u_{i,k}-\bar u_i)\varphi\,d \mu_{i,k} + (u_{i,k}-\bar u_i)\varphi\,d \bar \mu_i  \right|\leq  \|u_{i,k}-\bar u_i\|_{L^\infty(B_{2R}(0))} \left(\mu_{i,k}(B_{2R}(0)) +\bar \mu_i (B_{2R}(0)) \right)\rightarrow 0  $$
\end{proof}

\begin{proof}[End of the proof of Theorem \ref{teo:blow_up_convergence}.]  After Lemmata \ref{lemma:blow_up_convergence_1}--\ref{lemma:blow_up_convergence_7} the only thing left to prove are the claims that the measures $\mu_i$ are concentrated on $\Gamma_{\bar U}$ (for $i=1,\ldots,h$) and that property (G3) holds with $F\equiv 0$.

As for the first statement we start by fixing an $R>0$ and by considering a cut-off function $\varphi$ equal to one in $B_R(0)$, zero outside $B_{2R}(0)$. Since
$$
\int_{B_{2R}(0)}U_k \varphi\, d\mu_{i,k}=\int_{B_{2R}(0)\cap \Gamma_{U_k}}U_k \varphi \, d \mu_{i,k}=0,
$$
then
$$
0=\lim_k \int_{B_{2R}(0)}U_k \varphi \, d \mu_{i,k}=\lim_k \int_{B_{2R}(0)}(U_k-\bar U)\varphi \, d \mu_{i,k} + \lim_k \int_{B_{2R}(0)}\bar U \varphi\, d\mu_{i,k}=\int_{B_{2R}(0)}\bar U \varphi \, d\mu_i.
$$
Thus $\int_{B_R(0)}\bar U\, d \mu_i =0$ for every $R>0$ and in particular $\bar \mu_i(K)=0$ for every compact set $K\subset \R^N\setminus \Gamma_{\bar U}$, which proves the first claim.

As for the proof of the second claim, we recall that $U_k\in \Geh((\Omega-x_k)/t_k)$ and hence for any given $0<r_1<r_2$ the following equality holds
\begin{equation}\label{eq:Delta_of_E_U_k}
E(x_0,U_k,r_2)-E(x_0,U_k,r_1)=\int_{r_1}^{r_2}\left(\frac{2}{r^{N-2}}\int_{\partial B_r(x_0)}(\partial_\nu U_k)^2 \, d\sigma \right)\, dr + \int_{r_1}^{r_2}R(x_0,U_k,r)\, dr.
\end{equation}
Since $|\langle F_k(U_k), U_k \rangle | \leq d t_k^2 |U_k|^2\rightarrow 0$, we obtain
$$
E(x_0,U_k,r)=\frac{1}{r^{N-1}}\int_{B_r(x_0)}\left( |\nabla U_k|^2 - \langle F_k(U_k),U_k \rangle \right) \mathop{\rightarrow}_{k} \frac{1}{r^{N-1}}\int_{B_r(x_0)}|\nabla \bar U|^2=E(x_0,\bar U,r)
$$
for each fixed $r>0$. Moreover,

\begin{multline*}
\left|\int_{r_1}^{r_2}R(x_0,U_k,r)\,dr\right|\leq \ \left| \int_{r_1}^{r_2} \left( \frac{2}{r^{N-1}} \int_{B_r(x_0)} 
\sum_i f_i(x,u_{i,k})\langle \nabla u_{i,k},x-x_0\rangle \right)\, dr \right|+ \\
  \qquad  +\left| \int_{r_1}^{r_2} \left( \frac{1}{r^{N-1}}\int_{B_r(x_0)} (N-2)\langle F_k(U_k),U_k\rangle \right)\, dr \right|
+\left| \int_{r_1}^{r_2} \left( \frac{1}{r^{N-2}}\int_{\partial B_r(x_0)} \langle F_k(U_k),U_k\rangle\, d\sigma \right)\, dr \right|\\
\leq C(r_1,r_2)t_k^2\int_{B_{r_2}(x_0)} \sum_i u_{i,k}|\nabla u_{i,k}| + C(r_1,r_2)t_k^2\int_{B_{r_2}(x_0)}U_k^2\rightarrow 0.\qquad \qquad \quad
\end{multline*}

Finally, the fact that $U_k\rightarrow U$ strongly in $H^1_\text{loc}$ implies, up to a subsequence of $\{U_k\}$, that there exists a function $h(\rho)\in L^1(r_1,r_2)$ such that $\int_{\partial B_\rho(x_0)}|\nabla (U_k-\bar U)|^2\,d\sigma \leq h(\rho)$, and moreover $\int_{\partial B_\rho(x_0)}|\nabla (U_k-\bar U)|^2\,d\sigma\rightarrow 0$ for a.e. $\rho\in (r_1,r_2)$. Thus
$$
\int_{r_1}^{r_2}\left(\frac{2}{r^{N-2}}\int_{\partial B_r(x_0)}(\partial_\nu U_k)^2 \, d\sigma \right)\, dr \rightarrow 
\int_{r_1}^{r_2}\left(\frac{2}{r^{N-2}}\int_{\partial B_r(x_0)}(\partial_\nu \bar U)^2 \, d\sigma \right)\, dr .
$$
We can now pass to the limit in \eqref{eq:Delta_of_E_U_k} as $k\rightarrow +\infty$, obtaining
$$
E(x_0,\bar U,r_2)-E(x_0,\bar U,r_1)=\int_{r_1}^{r_2}\left( \frac{2}{r^{N-2}}\int_{\partial B_r(x_0)} (\partial_\nu \bar U)^2 \,d\sigma \right)\, dr,
$$
{\it i.e.}, (G3) holds for $\bar U$ with $F\equiv 0$.
\end{proof}

Up to now we have dealt with blowup sequences with arbitrary moving centers $\{x_k\}$. Next we observe that some particular choices of $x_k$ provide additional informational on the limit $\bar U$. More precisely, we have

\begin{coro}\label{coro:U_homogeneous}
Under the previous notations, suppose that one of these situations occurs:
\begin{enumerate}
\item $x_k=x_0$ for every $k$,
\item $x_k\rightarrow x_0 \in \Gamma_{U}$ and $N(x_0,U,0^+)=1$.
\end{enumerate}
Then $N(0,\bar U,r)=N(x_0,U,0^+)=:\alpha$ for every $r>0$, and $\bar U=r^\alpha G(\theta)$, where $(r,\theta)$ are the generalized polar coordinates centered at the origin.
\end{coro}
\begin{proof} We divide the proof in three steps.

\noindent STEP 1. $N(0,\bar U,r)$ is constant.

First observe that $N(0,U_k,r)=N(x_k,U,t_k r)$ and that Theorem \ref{teo:blow_up_convergence} yields $\lim_k N(0,U_k,r)=N(0,\bar U,r).$ As for the right hand side, if $x_k=x_0$ for some $x_0$, then $\lim_k N(x_0,U,t_k r)=N(x_0,U,0^+)$ for every $r>0$ by Theorem \ref{teo:Almgren's_Monotonicity_Formula}. In the second situation we claim that $\lim_k N(x_k,U,t_k r)=1$. Denoting by $\tilde r$ the radius associated to $\tilde \Omega$ in the context of Theorem \ref{teo:Almgren's_Monotonicity_Formula}, for any given $\varepsilon>0$ take $0<\bar r=\bar r(\varepsilon)\leq \tilde r$ such that
$$
N(x_0,U,r)\leq 1+\frac{\varepsilon}{2} \qquad \text{ for every } 0<r\leq \bar r, \qquad \text{ and }\qquad  e^{\tilde C \bar r}\leq \frac{2+2\varepsilon}{2+\varepsilon}.
$$
Moreover there exists $r_0>0$ such that
$$
N(x,U,\bar r)\leq 1+\varepsilon \qquad \text{ for } x\in B_{r_0}(x_0)\subseteq \tilde \Omega.
$$
Thus, again by Theorem \ref{teo:Almgren's_Monotonicity_Formula}, we obtain
$$
N(x,U,r)\leq (2+\varepsilon)e^{\tilde C\bar r}-1\leq 1+2\varepsilon, \qquad \quad \text{ for every } x\in B_{r_0}(x_0),\ 0<r\leq \bar r,
$$
and the claim follows by also taking into account Corollary \ref{coro:N_geq_1}.

 \noindent STEP 2. The derivative of $N$.

An easy computation gives
$$
\frac{d}{dr}H(0,\bar U,r)=\frac{2}{r^{N-1}}\int_{\partial B_r(x_0)} \langle \bar U, \partial_\nu \bar U\rangle \, d\sigma \qquad \qquad \text{for a.e. } r>0
$$
which together with identity \eqref{eq:derivative_of_E_limit} - for $y_0=0$ - readily implies
$$
0=\frac{d}{dr}N(0,\bar U,r)=\frac{2}{r^{2N-3}H^2(0,\bar U,r)}\left\{\int_{\partial B_r(0)}\bar U^2 \, d\sigma \int_{\partial B_r(0)}(\partial_\nu \bar U)^2\, d\sigma - \left( \int_{\partial B_r(0)} \langle \bar U, \partial_\nu \bar U\rangle \, d\sigma  \right)^2  \right\}
$$
for a.e. $r>0$.

\noindent  STEP 3. $U$ is homogeneous.

The previous equality yields the existence of $C(r)>0$ such that $\partial_\nu \bar U=C(r)\bar U$ for a.e. $r>0$. By using this information in \eqref{eq:derivative_of_log_H} we get
$$
2 C(r)=\frac{  2\int_{\partial B_r(0)} \langle \bar U , \partial_\nu \bar U \rangle \, d\sigma }{\int_{\partial B_r(0)}\bar U^2 \, d\sigma }=\frac{d}{dr} \log(H(0,\bar U,r))=\frac{2}{r}\alpha,
$$
and thus $C(r)=\alpha/r$ and $\bar U(x)=r^\alpha G(\theta)$.
\end{proof}
%

\section{Hausdorff dimension estimates for nodal and singular sets}\label{sec:hausdorff_dimension}

As we mentioned before, our main interest is the study of the free boundary $\Gamma_U=\{x\in \Omega:\ U(x)=0\}$ for every $U\in \Geh(\Omega)$. As a first step in its characterization we will provide an estimate of its Hausdorff dimension. Regarding its regularity, we shall decompose $\Gamma_U$ in two parts:
\begin{itemize}
\item the first one - which will be denoted by $S_U$ - where we are not able to prove any kind of regularity result, but which has a ``small'' Hausdorff dimension,
\item the second one - $\Sigma_U$ - where we are able to prove regularity results ({\it cf.} Theorem \ref{teo:main_result}).
\end{itemize}

\begin{defin}\label{def:singular_regular_set}
Given $U\in\Geh(\Omega)$ we define its regular and singular sets respectively by
$$\Sigma_U=\{x\in \Gamma_U:\ N(x,U,0^+)=1\}, \quad \text{ and } \quad S_U=\Gamma_U\setminus \Sigma_U=\{x\in \Gamma_U:\ N(x,U,0^+)>1\}.$$
\end{defin}

In the same spirit of \cite[Lemma 4.1]{CL} we prove that there exists a jump in the possible values of $N(x_0,U,0^+)$ for $x_0\in \Gamma_{U}$ (recall that $N(x_0,U,0^+)\geq 1$ by Corollary \ref{coro:N_geq_1}). In \cite{CL}, the authors deal with solutions of minimal energy, proving directly the existence of a jump in any dimension. In our general framework their strategy does not work; instead, we will obtain the same results via an iteration procedure. In the following proposition we start to prove the existence of a jump in dimension $N=2$. The extension to higher dimensions will be treated in the subsequent sections.

\begin{prop}\label{prop:N=1_or_N_geq_1+delta}
Let $N=2$. Given $U\in \Geh(\Omega)$ and $x_0\in \Gamma_U$, then either
$$N(x_0,U,0^+)=1 \qquad \qquad  \text{ or } \qquad \qquad N(x_0,U,0^+)\geq 3/2.$$
\end{prop}
\begin{proof}
We perform a blowup at $x_0$ by considering $U_k(x)=U(x_0+t_k x)/ \rho_k$,
where $\rho_k=\|U(x_0+t_k \cdot)\|_{L^2(\partial B_1(0))}$  and  $t_k\downarrow 0$ is an arbitrary sequence. Theorem \ref{teo:blow_up_convergence} together with Corollary \ref{coro:U_homogeneous} (case 1) yield the existence of $\bar U=r^\alpha G(\theta)\in \Geh_\text{loc}(\R^N)$ with $\alpha=N(x_0,U,0^+)$ such that (up to a subsequence) $U_k\rightarrow \bar U$ strongly in $H^1_\text{loc}\cap C^{0,\beta}_\text{loc}(\R^N)$ for every $0<\beta<1$. Moreover, each component $\bar u_i$ is harmonic in the open set $\{\bar u_i>0\}$, which implies that on every given connected component $A\subseteq \{g_i>0\}\subseteq \partial B_1(0)$ it holds
$$- g_i'' (\theta)= \lambda g_i (\theta), \qquad \qquad \text{with }\lambda=\alpha^2.$$
In particular $\lambda=\lambda_1(A)$ (the first eigenvalue) because $g_i\geq 0$ and $g_i\not\equiv 0$, and moreover $\lambda_1(\cdot)$ has the same value on every connected component of $\{G>0\}$.

Suppose that $\{G>0\}$ has at least three connected components. Then one of them, denote it by $C$, must satisfy $\Hh^{1}(C)\leq \Hh^{1}(\partial B_1(0))/3$. By using spherical symmetrization (Sperner's Theorem) and the monotonicity of the first eigenvalue with respect to the domain, we obtain
$$\lambda=\lambda_1(C)\geq \lambda_1\left(E\left(\pi/3\right)\right),\quad \text{ where } E\left(\pi/3\right)=\{x\in \partial B_1(0):\ arcos(\langle x,e_3 \rangle)<\pi/3\}$$ ($e_3=(0,0,1)$). Since $\lambda_1(E(\pi/3))=(3/2)^2$ with eigenfunction cos$(3\theta/2)$ - in polar coordinates - we deduce $\alpha\geq 3/2$.

Suppose now that $\{G>0\}$ has at most two connected components. Since $N=2$ and $\{\bar U=0\}$ has an empty interior (Remark \ref{rem:Gamma_empty_interior}), then the number of components is equal to the number of zeros of $G$ on $\partial B_1(0)$. Moreover $G$ must have at least one zero, because otherwise $G>0$ on $\partial B_1(0)$, $\bar U$ is harmonic in $\R^2\setminus\{0\}$ and hence $\bar U\equiv 0$ (recall that $\bar U(0)=0$), a contradiction. If $G$ has one single zero then $\lambda=\lambda_1(E(\pi))=1/4$ and $\alpha=1/2$, contradicting Corollary \ref{coro:N_geq_1}. Hence we have concluded that $G$ must have exactly two zeros. Denote by $\Omega_1$ and $\Omega_2$ the two connected components of $\{G>0\}$. Since $\lambda_1(\Omega_1)=\lambda_1(\Omega_2)$, $\Omega_1$ and $\Omega_2$ must cut the sphere in two equal parts and thus $\lambda=\lambda_1(E(\pi/2))=1$, $\alpha =1$.
\end{proof}

\begin{coro}\label{coro:S_U_closed}
For $N=2$ the set $S_U$ is closed in $\Omega$, whenever $U\in \Geh(\Omega)$.
\end{coro}
\begin{proof}
This is a direct consequence of Proposition \ref{prop:N=1_or_N_geq_1+delta} together with the upper semi-continuity of the function $x\mapsto N(x, U, 0^+)$ stated in Corollary \ref{coro:N_upper_semi_continuous}.
\end{proof}

Moreover a careful examination of the proof of Proposition \ref{prop:N=1_or_N_geq_1+delta} provides a more detailed description of the blowup limits:

\begin{rem}\label{rem:N=2_N=1_nodalsetishyper-plane}
Let $N=2$ and let $\bar U$ be a blowup limit under the hypotheses of Corollary \ref{coro:U_homogeneous}. Then $\{\bar U>0\}$ has at least three connected components if and only if $\alpha=N(x_0,U,0^+)>1$. If on the other hand $\alpha=N(x_0,U,0^+)=1$ then $\{\bar U>0\}$ is made of exactly two connected components and $\Gamma_{\bar U}$ is an hyper-plane (more precisely, denoting by $\nu$ a normal vector of $\Gamma_{\bar U}$, then on one side of $\Gamma$ the non trivial component of $\bar U$ is equal to $a_1(x\cdot \nu)^+$, and on the other equals $a_2(x\cdot \nu)^-$, for some $a_1, a_2>0$).
\end{rem}

Next we state and prove some estimates regarding the Hausdorff dimensions of the sets under study. The following result implies part of Theorem \ref{teo:main_result}.

\begin{teo}\label{teo:Hausdorff_dim_estimates}
Let $U\in\Geh(\Omega)$. Then
\begin{enumerate}
\item $\Hh_\text{dim}(\Gamma_U)\leq N-1$ for any $N\geq 2$.
\item$\Hh_\text{dim}(S_U)=0$ for $N=2$, and moreover for any given compact set $\tilde \Omega\Subset \Omega$ we have that $S_U\cap \tilde \Omega$ is a finite set.
\end{enumerate}
\end{teo}

For the moment the second statement holds only for $N=2$ because of the dimension restriction in Proposition \ref{prop:N=1_or_N_geq_1+delta} (which provides the closedness of $S_U$). As we said before we shall extend ahead these results to any dimension greater than or equal to two.

The rest of this section is devoted to the proof of this result. The idea is to apply a version of the so called Federer's Reduction Principle, which we now state.

\begin{teo}\label{teo:FRP}
Let $\mathcal{F}\subseteq (L^\infty_\text{loc}(\R^N))^h$, and define, for any given $U\in\Feh,\ x_0\in\R^N$ and $t>0$, the rescaled and translated function $$U_{x_0,t}:=U(x_0+t\cdot).$$
We say that $U_n\rightarrow U$ in $\Feh$ iff $U_n\rightarrow U$ uniformly on every compact set of $\R^N$.
Assume that $\mathcal{F}$ satisfies the following conditions:
\begin{itemize}
\item[(A1)] (Closure under rescaling, translation and normalization)
Given any $|x_0|\leq 1-t, 0<t<1$, $\rho>0$ and $U\in \Feh$, we have that also $\rho\cdot U_{x_0,t}\in \Feh$.
\item[(A2)] (Existence of a homogeneous ``blow--up'')
Given $|x_0|<1, t_k\downarrow 0$ and $U \in \Feh$, there exists a sequence $\rho_k\in(0,+\infty)$, a real number $\alpha\geq 0$ and a  function $\bar U\in \Feh$ homogeneous of degree\footnote{That is, $\bar U(tx)=t^\alpha U(x)$ for every $t>0$.} $\alpha$ such that, if we define $U_k(x)=U(x_0+t_k x)/\rho_k$, then
$$U_k\rightarrow \bar U \qquad {\rm in }\ \Feh,\qquad \qquad \text{ up to a subsequence}.$$
\item[(A3)] (Singular Set hypotheses)
There exists a map $\Ss:\Feh\rightarrow \Ceh$ (where $\Ceh:=\{A\subset \R^N:\ A\cap B_1(0) \text{ is relatively closed in } B_1(0)\}$) such that
\begin{itemize}
 \item [(i)]  Given $|x_0|\leq 1-t$, $0<t<1$ and $\rho>0$, it holds $$\Ss(\rho\cdot U_{x_0,t})=(\Ss(U))_{x_0,t}:=\frac{\Ss(U)-x_0}{t}.$$
 \item[(ii)] Given $|x_0|<1$, $t_k\downarrow 0$ and $U,\bar U\in \Feh$ such that there exists $\rho_k>0$ satisfying $U_k:=\rho_k U_{x_0,t_k}\rightarrow \bar U$ in $\Feh$, the following ``continuity'' property holds:
$$\forall \varepsilon>0\ \exists k(\epsilon)>0:\ k\geq k(\varepsilon) \Rightarrow \Ss(U_k)\cap B_1(0)\subseteq \{x\in \R^N:\ \text{dist}(x,\Ss(\bar U))<\varepsilon\}.$$
\end{itemize}
\end{itemize}
Then, if we define
\begin{multline}\label{eq:FRP_definition_of_d}
d=\max \left\{ {\rm dim }\ L:\ L \text{ is a vector subspace of } \R^N \text{ and there exist } U\in \Feh \text{ and }\a\geq0 \right.\\
\left. \text{ such that } \Ss(U)\neq \emptyset \text{ and } U_{y,t}=t^\a U\ \forall y\in L,\ t>0 \right\} ,
\end{multline}
either $\Ss(U)\cap B_1(0)=\emptyset$ for every $U\in \Feh$, or else $\Hh_{\rm dim}(\Ss(U)\cap B_1(0))\leq d$ for every $U\in\Feh$. Moreover in the latter case there exist a function $V\in \Feh$, a d-dimensional subspace $L\leq \R^N$ and a real number $\a\geq0$ such that
\begin{equation*}\label{Psi_invariant_over_L}
V_{y,t}=t^\a V \qquad \forall y\in L, \ t>0, \qquad \quad \text{ and } \qquad \quad \Ss(V)\cap B_1(0)=L\cap B_1(0).
\end{equation*}
If $d=0$ then $\Ss(U)\cap B_\rho(0)$ is a finite set for each $U\in \Feh$ and $0<\rho<1$.
\end{teo}

Up to our knowledge, this principle (due to Federer) appeared in this form for the first time in the book by Simon \cite[Appendix A]{Simon}. The version we present here can be seen as a particular case of a generalization made by Chen (see \cite[Theorem 8.5]{Chen_1} and \cite[Proposition 4.5]{Chen_2}).

\begin{proof}[Proof of Theorem \ref{teo:Hausdorff_dim_estimates}]
A first observation is that we only need to prove that the Hausdorff dimension estimates of the theorem hold true for the sets $\Gamma_U\cap B_1(0)$ and $S_U\cap B_1(0)$ whenever $U\in \Geh(\Omega)$ with $B_2(0)\Subset \Omega$. In fact, if we prove so, then we obtain that for any given $\Omega$ and $U\in\Geh(\Omega)$ it holds $\Hh_\text{dim}(\Gamma_U\cap K)\leq N-1$, $\Hh_\text{dim}(S_U\cap K)\leq N-2$ for every $K\Subset \Omega$ (because rescaling a function does not change the Hausdorff dimension of its nodal and singular sets). Being this true the theorem follows because a countable union of sets with Hausdorff dimension less than or equal to some $n\in\R^+_0$ also has Hausdorff dimension less than or equal to $n$.

Thus we apply the Federer's Reduction Principle to the following class of functions
$$
\Feh=\{U\in \left(L^\infty_{\rm loc}(\R^N)\right)^h:\ \text{there exists some domain } \ \Omega \text{ such that } B_2(0)\Subset\Omega \text{ and } U_{|_\Omega}\in \Geh(\Omega)\}.
$$
Let us start by checking (A1) and (A2). Hypothesis (A1) is immediately satisfied by Proposition \ref{prop:invariance_under_translation}. Moreover, let $|x_0|<1$, $t_k\downarrow 0$ and $U\in \Feh$, and choose $\rho_k=\| U(x_0+t_k x) \|_{L^2(\partial B_1(0))}$. Theorem \ref{teo:blow_up_convergence} and Corollary \ref{coro:U_homogeneous} (case 1) yield the existence of $\bar U\in \Feh$ such that (up to a subsequence) $U_k\rightarrow \bar U$ in $\Feh$ and $\bar U$ is a homogeneous function of degree $\alpha=N(x_0,U,0^+)\geq 0$. Hence also (A2) holds. Next we choose the map $\Ss$ according to our needs.

1. (dimension estimate of the nodal sets in arbitrary dimensions) We want to prove that $\Hh_{\rm dim}(\Gamma_U\cap B_1(0))\leq N-1$ whenever $U\in\Feh$. Define $\Ss: \Feh\rightarrow \Ceh$ by $\Ss(U)=\Gamma_U$ ($\Gamma_U\cap B_1(0)$ is obviously closed in $B_1(0)$ by the continuity of $U$). It is quite straightforward to check hypothesis (A3)-(i), and the local uniform convergence considered in $\Feh$ clearly yields (A3)-(ii). Therefore, in order to end the proof in this case the only thing left to prove is that the integer $d$ associated to $\Ss$ (defined in \eqref{eq:FRP_definition_of_d}) is less than or equal to $N-1$. Suppose by contradiction that $d=N$; then this would imply the existence of $V\in \Feh$ with $\Ss(V)=\R^N$, {\it i.e}., $V\equiv 0$, which contradicts the definition of $\Geh$. Thus $d\leq N-1$.

2. (dimension estimate of the singular sets in the case $N=2$) This is the most delicate case. As we said before, the restriction of $N$ is only due to Proposition \ref{prop:N=1_or_N_geq_1+delta}. As we shall see, the rest of the argument does not depend on the chosen dimension; for this reason, and since moreover we will prove the closedness of $S_U$ for any dimension $N\geq 2$ in Section \ref{sec:iteration_argument}, we decide to keep $N$ in the notations. We define $\Ss: \Feh\rightarrow \Ceh$ by $\Ss(U)=S_U$ (which belongs to $\Ceh$ by Corollary \ref{coro:S_U_closed}). The map satisfies (A3)-(i) thanks to identity \eqref{eq:almgren_identities2}, more precisely
$$
x\in \Ss(U_{x_0,t}/\rho)\Leftrightarrow N(x,U_{x_0,t}/\rho,0^+)>1\Leftrightarrow N(x_0+t x,U,0^+)>1\Leftrightarrow x_0+t x\in \Ss(U).
$$

As for (A3)-(ii), take $U_k,U\in \Feh$ as stated. Then in particular $U_k\rightarrow U$ uniformly in $B_2(0)$ and by arguing as in the proof of Lemma \ref{lemma:blow_up_convergence_7} it is easy to obtain strong convergence in $H^1(B_{3/2}(0))$. Suppose now that (A3)-(ii) does not hold; then there exists a sequence $x_k\in B_1(0)$ ($x_k\rightarrow x$, up to a subsequence, for some $x$) and an $\bar \varepsilon>0$ such that $N(x_k,U_k,0^+)\geq 1+\delta$ and $\text{dist}(x_k,\Ss(U))\geq \bar\varepsilon $. But then for small $r$ we obtain (as in the proof of Corollary \ref{coro:N_upper_semi_continuous})
$$
N(x_k,U_k,r)\geq (2+\delta) e^{-Cr}-1,
$$
and hence (since $N(x_k,U_k,r)\rightarrow N(x,U,r)$ in $k$ for small $r$) $N(x,U,0^+)\geq 1+\delta $, a contradiction.

 Finally let us prove that $d\leq N-2$. If $d=N-1$ then we would have the existence of a function $V$, homogeneous with respect  to every point in\footnote{For some $\alpha>0$ we have $V(y+\lambda x)=\lambda^\alpha V(x)$ for every $y\in \R^{N-1}\times \{0\}$, $x\in \R^N$.} $\R^{N-1}\times\{0\}$
  such that $S_V=\R^{N-1}\times\{0\}$. Now, if we take a usual blowup sequence centered at $x_0=0$ ($V(t_k x)/\rho_k$), we obtain at the limit a function $\bar U=r^\alpha G(\theta)\in \Geh_\text{loc}(\R^N)$ with $\alpha=N(x_0,V,0^+)>1$, harmonic in $\R^N\setminus \Gamma_{\bar U}$ such that $\bar U(y+\lambda x) = \lambda^\alpha \bar U(x)$ whenever $y\in \R^{N-1}\times \{0\}$, $x\in \R^N$. We prove that $\Gamma_{\bar U}=\R^{N-1}\times\{0\}$, which leads to a contradiction since Hopf's Lemma implies $\alpha=1$. Since $\bar U(x)=\lim V(t_k x)/\rho_k$ and $\Gamma_V=\R^{N-1}\times\{0\}$, it is obvious that  $\R^{N-1}\times\{0\}\subseteq \Gamma_{\bar U}$. If there were $y\in \Gamma_{\bar U}\setminus (\R^{N-1}\times\{0\})$, then since $\bar U$ is homogeneous with respect to every point in $\R^{N-1}\times\{0\}$, we would have that either $\R^{N-1}\times[0,+\infty)$ or $\R^{N-1}\times(-\infty,0]$ would be contained in $\Gamma_{\bar U}$, contradicting Remark \ref{rem:Gamma_empty_interior}.  \end{proof}

\begin{rem}\label{rem:what_do_we_need_to_have_hausdorff_estimates}
The proof of Theorem \ref{teo:Hausdorff_dim_estimates}-2 would hold in arbitrary dimensions provided that for every $N\geq 2$ there exists an universal constant $\delta_N>1$ such that either $N(x_0,U,0^+)=1$ or $N(x_0U,0^+)\geq \delta_N$, whenever $U\in \Geh(\Omega)$ and $x_0\in \Gamma_U$. A careful examination of the proof of Proposition \ref{prop:N=1_or_N_geq_1+delta} shows that the latter statement is equivalent to the following one:
\begin{itemize}
\item for every $\bar U=r^\alpha G(\theta)\in \Geh_\text{loc}(\R^N)$ with $\Delta \bar U=0$ in $\{\bar U>0\}$, either $\alpha=1$ or $\alpha\geq 1+\delta_N$.
\end{itemize}
\end{rem}


\section{Regularity results under a flatness-type assumption}\label{sec:regularity_of_regular_free_boundary}

This section is devoted to the proof of the following auxiliary result.

\begin{teo}\label{teo:Sigma_U_smooth}
Let $\Omega$ be a domain in $\R^N$ with $N\geq 2$. Fix $U\in \Geh(\Omega)$ and let $\Gamma^\star$ be a relatively open subset of $\Gamma_U$ such that the following property holds:
$$
(P)
\begin{array}{l}
\text{For any } x_0\in \Gamma^\star \text{ take } x_k\rightarrow x_0,\ t_k\downarrow 0, \text{and } \bar U\in \Geh_\text{loc}(\R^N)  \text{ such that } \\
 \bar U=\lim_k U(x_k+t_k x)/\rho_k  \text{ with } \rho_k=\|U(x_k+t_k \cdot) \|_{L^2(\partial B_1(0))}.\\
\text{Then } \Gamma_{\bar U} \text{ is a hyper-plane passing through the origin. }
\end{array}
$$
Then $\Gamma^\star$ is a $C^{1,\alpha}$ hyper-surface for every $0<\alpha<1$ and for every $x_0\in \Gamma^\star$
\begin{equation}\label{eq:reflection_principle_section5}
\lim_{x\rightarrow x_0^+}|\nabla U(x)|= \lim_{x\rightarrow x_0^-} |\nabla U(x)|\neq 0,
\end{equation}
where the limits represent an approximation to $x_0$ coming from opposite sides of the hyper-surface.
\end{teo}

\begin{rem}\label{rem:U_holds_in_dim_2}
In dimension $N=2$, for every $U\in \Geh(\Omega)$, property (P) holds for $\Gamma^\star:=\Sigma_U$, as previously observed in Remark \ref{rem:N=2_N=1_nodalsetishyper-plane}.
\end{rem}

In general, Theorem \ref{teo:blow_up_convergence} yields that every blowup limit $\bar U$ belongs to $\Geh_\text{loc}(\R^N)$ and that $-\Delta \bar u_i=\bar \mu_i$, with $\bar \mu_i\in \Mah_\text{loc}(\R^N)$ non negative and concentrated on $\Gamma_{\bar U}$. Property (P) says that such nodal sets are ``flat'', whenever the blowup limit is taken at points of $\Gamma^*$. Hence Theorem \ref{teo:Sigma_U_smooth} states that ``locally flat'' points of the free boundary $\Gamma_U$ (for $U\in \Geh(\Omega)$) are regular and that a reflection law holds. The previous theorem will be an important tool in the proof of Theorem \ref{teo:main_result} (this will became clear in Section \ref{sec:iteration_argument} ahead): we will be able to apply this result to $\Sigma_U$ in any dimension $N\geq 2$.

The strategy of the proof of Theorem \ref{teo:Sigma_U_smooth} is as follows: property (P) will provide a local separation property (Proposition \ref{prop:local_separation_property}). This, together with the fact that $\bar U\in \Geh_\text{loc}(\R^N)$ will allow us the use of a reflection principle (Lemma \ref{lemma:reflection_principle}), which will in turn imply that in a small neighborhood of each point in $\Gamma^\star$ a certain equation can be solved and has a $C^{1,\alpha}$ solution (Theorem \ref{teo:global_equation}). The nodal set of this solution will be equal to $\Gamma_U$, and the final step will be to establish that its gradient is non zero on $\Gamma_U$.

From now on we fix $U\in \Geh(\Omega)$ with $\Omega\subseteq \R^N$ ($N\geq 2$) and let $\Gamma^\star$ be a relatively open subset  of $\Gamma_U$ satisfying assumption (P). Take an open set $\tilde \Omega \Subset \Omega$ such that $\Gamma_U\cap \overline{\tilde \Omega} = \Gamma^\star\cap \overline{\tilde\Omega}$, that is, all the nodal points of $U$ in the closure of $\tilde \Omega$ belong to $\Gamma^\star$. In the following lemma we prove that $\Gamma_U\cap \tilde \Omega$ verifies the so called $(N-1)$--dimensional $\delta$--Reifenberg flat condition for every $0<\delta<1$.

\begin{lemma}\label{lemma:flatness_condition}
Within the previous framework, for any given $0<\delta<1$ there exists $R>0$ such that for every $x\in \Gamma^\star\cap \tilde \Omega=\Gamma_U\cap\tilde \Omega$ and $0<r<R$ there exists an hyper-plane $H=H_{x,r}$ containing $x$ such that \footnote{Here $d_\Hh(A,B):=\max\{\sup_{a\in A}\text{dist}(a,B),\sup_{b\in B}\text{dist}(A,b)\}$ denotes the Hausdorff distance. Notice that $d_\Hh(A,B)\leq \delta$ if and only if $A\subseteq N_\delta(B)$ and $B\subseteq N_\delta (A)$, where $N_\delta(\cdot)$ is the closed $\delta$--neighborhood of a set.}
\begin{equation}\label{eq:flatness_condition}
d_\Hh (\Gamma_{U}\cap B_r(x),H\cap B_r(x))\leq \delta r.
\end{equation}
\end{lemma}

\begin{proof}
Arguing by contradiction, suppose there exist $\bar \delta >0$ and subsequences $x_n\in \Gamma^\star\cap\tilde \Omega$, $r_n\rightarrow 0$ such that
\begin{equation*}
d_\Hh (\Gamma_U\cap B_{r_k}(x_n),H\cap B_{r_k}(x_k)) > \bar \delta r_k.
\end{equation*}
whenever $H$ is an hyper-plane passing through $x_k$. If we take a blowup sequence of type $U_k(x)=U(x_k+r_k x)/\rho_k$ (here we use the notations of Section \ref{sec:blow_up_sequences}), then the contradiction statement is equivalent to have
\begin{equation*}
d_\Hh(\Gamma_{U_k}\cap B_1(0),H\cap B_1(0))>\bar \delta
\end{equation*}
whenever $H$ is an hyper-plane that passes through the origin. Since, up to a subsequence, $x_k\rightarrow \bar x\in \Gamma_U\cap \overline{\tilde \Omega}=\Gamma^\star\cap \overline{\tilde \Omega}$, Theorem \ref{teo:blow_up_convergence} together with property (P) implies the existence of a blowup limit $\bar U$ whose nodal set $\Gamma_{\bar U}$ is a hyper-plane containing the origin.
Hence we obtain a contradiction once we are able to prove that
$$
d_\Hh(\Gamma_{U_k}\cap B_1(0),\Gamma_{\bar U}\cap B_1(0))\rightarrow 0.
$$
\noindent
i) For every $\varepsilon>0$ there exists $\bar k>0$ such that
$$
\Gamma_{U_k}\cap B_1(0)\subseteq N_\varepsilon(\Gamma_{\bar U}\cap B_1(0)) \qquad \text{ for every } k\geq \bar k.
$$
Were the previous inclusion not true and we would obtain the existence of $\bar \varepsilon>0$ and of a sequence $y_k\in \Gamma_{U_k}\cap B_1(0)$ such that $\text{dist}(y_k,\Gamma_{\bar U}\cap B_1(0))>\bar \varepsilon$. Up to a subsequence, $y_k\rightarrow y\in \Gamma_{\bar U}\cap \bar B_1(0)$ by the $L^\infty_\text{loc}$ convergence $U_k\rightarrow \bar U$; moreover, since $\Gamma_{\bar U}$ is a hyper-plane passing  the origin, we deduce that $\text{dist}(y,\Gamma_{\bar U}\cap B_1(0))=0$, which provides a contradiction.

\noindent
ii) For every $\varepsilon>0$ there exists $\bar k>0$ such that
\begin{equation}\label{eq:flat_cond_inclusion}
\Gamma_{\bar U}\cap B_1(0)\subseteq N_\varepsilon(\Gamma_{U_k}\cap B_1(0)) \qquad \text{ for every } k\geq \bar k.
\end{equation}
First of all we prove that given $x\in\Gamma_{\bar U}$ and $\delta>0$, $U_k$ must have a zero in $B_\delta(x)$ for large $k$. If not, by recalling that $u_{i,k}\cdot u_{j,k}\equiv 0$ whenever $i\neq j$, we would have $u_{i,k}>0$ in $B_\delta(x)$ for some $i$ and moreover $\Delta u_{i,k}=0$ and $u_{j,k}\equiv 0$ (for $j\neq i$) in such ball. This would imply $\bar u_j\equiv 0$, $\Delta \bar u_i=0$ in $B_\delta(x)$ with $x\in \Gamma_{\bar U}$, and therefore $\bar U\equiv 0$ in $B_\delta(x)$, a contradiction by Remark \ref{rem:Gamma_empty_interior}.

Now we are in condition to prove \eqref{eq:flat_cond_inclusion}. We use once again a contradiction argument: suppose the existence of $\bar \varepsilon>0$ and $y_k\in \Gamma_{\bar U}\cap B_1(0)$, $y_k\rightarrow y\in \Gamma_{\bar U}\cap \bar B_1(0)$, such that $\text{dist}(y_k,\Gamma_{U_k}\cap B_1(0))>\bar \varepsilon$. Since $\Gamma_{\bar U}$ is a hyper-plane passing trough the origin, we can take $\bar y\in \Gamma_{\bar U}\cap B_1(0)$ such that $|y-\bar y|\leq \bar \varepsilon/4$. Moreover, by making use of the result proved in the previous paragraph, we can take a sequence $\bar y_k\in \Gamma_{U_k}\cap B_1(0)$ such that $|\bar y_k-\bar y|\leq \bar \varepsilon/4$ for large $k$. But then
$$
\text{dist}(y_k,\Gamma_{\bar U}\cap B_1(0))\leq |y_k-\bar y_k|\leq +|y_k-y|+| y- \bar y|+|\bar y-\bar y_k|\leq 3\bar \varepsilon/4<\bar \varepsilon
$$for large $k$, a contradiction.
\end{proof}

With the $(N-1)$--dimensional $\delta$--Reifenberg property we are able to prove a local separation result. We quote Theorem 4.1 in \cite{reifenberg_separates} for a result in the same direction.

\begin{prop}[Local Separation Property]\label{prop:local_separation_property}
Given $x_0\in \Gamma^\star$ there exists a radius $R_0>0$ such that $B_{R_0}(x_0)\cap \Gamma^\star=B_{R_0}(x_0)\cap \Gamma_U$ and $B_{R_0}(x_0)\setminus \Gamma_U=B_{R_0}(x_0)\cap \{U>0\} $ has exactly two connected components $\Omega_1,\Omega_2$. Moreover, for sufficiently small $\delta>0$, we have that given $y\in \Gamma_U\cap B_{R_0}(x_0)$ and $0<r<R-|y|$ there exist a hyper-plane $H_{y,r}$ (passing through $y$) and a unitary vector $\nu_{y,r}$ (orthogonal to $H_{y,r}$) such that
$$
\{x+t \nu_{y,r}\in B_r(y):\ x\in H_{y,r},\ t\geq \delta r\}\subset \Omega_1, \qquad \{x-t \nu_{y,r}\in B_r(y):\ x\in H_{y,r},\ t\geq \delta r\}\subset \Omega_2.
$$
\end{prop}
\begin{proof}
Let $s$ be such that $B_{2s}(x_0) \cap \Gamma^\star=B_{2s}(x_0)\cap \Gamma_U$ (which exists since $\Gamma^\star$ is a relatively open set in $\Gamma_U$) and fix $\delta<1/8$. With the notations of Lemma \ref{lemma:flatness_condition}, for $\tilde \Omega:= B_s(x_0)$ there exists $R>0$ such that $\Gamma_U\cap B_{s}(x_0)$ satisfies a $(\delta,R)$--Reifenberg flat condition. We show that Proposition \ref{prop:local_separation_property} holds with the choice $R_0:=\min\{R,s\}$.

Lemma \ref{lemma:flatness_condition} yields the existence of an hyper-plane $H_{x_0,R_0}$ containing the origin such that
\begin{equation}\label{eq:local_separation_property1}
d_\Hh (\Gamma_U\cap B_{R_0}(x_0),H_{x_0,R_0}\cap B_{R_0}(x_0))\leq \delta R_0.
\end{equation}
Thus the set $B_{R_0}(x_0)\setminus N_{\delta R_0}(H_{x_0,R_0})$ is made of two connected components, say $A_1$ and $A_2$, which do not intersect $\Gamma_U$. Define the function
$$
\sigma(x)=\left\{
\begin{array}{cll}
1 & \text{ if }  & x\in A_1,\\
-1 & \text{ if } & x\in A_2.
\end{array}
\right.
$$
Now take any point $x_1\in \Gamma_U\cap B_{R_0}(x_0)\subseteq N_{\delta R_0}(H_{x_0,R_0})\cap B_{R_0}(x_0)$ and consider a ball of radius $R_0/2$ centered at $x_1$. Once again by Lemma \ref{lemma:flatness_condition} we have the existence of an hyper-plane $H_{x_1,R_0/2}$ such that
$$
d_\Hh(\Gamma_U\cap B_{R_0/2}(x_1),H_{x_1,R_0/2}\cap B_{R_0/2}(x_1))\leq \delta R_0/2.
$$
This inequality together with \eqref{eq:local_separation_property1} yields that
$$
N_{\delta R_0/2}(H_{x_1,R_0/2})\cap B_{R_0/2}(x_1)\cap B_{R_0}(x_0)\subseteq N_{4\delta R_0}(H_{x_0,R_0})\cap B_{R_0}(x_0).
$$
Hence $B_{R_0}(x_0)\cap B_{R_0/2}(x_1)\setminus N_{\delta R_0/2}(H_{x_1,R_0/2})$ has exactly two connected components where one intersects $A_1$ but not $A_2$, and the other intersects $A_2$ but not $A_1$. Thus the set
$$
\left( \cup_{x_1\in \Gamma_U\cap B_{R_0}(x_0)} B_{R_0}(x_0)\cap B_{R_0/2}(x_1)\setminus N_{\delta R_0/2}(H(x_1,R_0/2))\right) \cup A_1\cup A_2
$$
has exactly 2 connected components which do not intersect $\Gamma_U$ and hence we can continuously extend (by $\pm 1$) the function $\sigma$ to this set.

Now we iterate this process: in the $k$--th step, we apply the previous reasoning to a ball of radius $R_0/2^k$ centered at a point of $\Gamma_U$. In this way we find two connected and disjoint sets $\Omega_1, \Omega_2$ such that $B_{R_0}(x_0)\setminus \Gamma_U=\Omega_1\cup \Omega_2$, $A_1\subseteq \Omega_1$, $A_2\subseteq \Omega_2$. Moreover, the function $\sigma:B_1(0)\setminus \Gamma_U\rightarrow \{-1,1\}$ defined by $\sigma(x)=1$ if $x\in \Omega_1$, $\sigma(x)=-1$ if $x\in \Omega_2$ is continuous and thus $B_{R_0}(x_0)\setminus \Gamma_U$ has exactly two connected components. In order to check the continuity, take $x\in B_{R_0}(x_0)$ such that $\text{dist}(x,\Gamma_U\cap B_{R_0}(x_0))=:\gamma>0$, let $\bar x\in \Gamma_U\cap B_{R_0}(x_0)$ be a point of minimum distance and take $k$ so large that $R_0/2^{k+1}<\gamma<R_0/2^k$; then $x\in B_{R_0/2^k}(\bar x)\setminus N_{\delta R_0/2^k}(H_{\bar x,R_0/2^k})$ and hence $\sigma$ is constant (recall the construction of this function) in a small neighborhood of $x$.
\end{proof}

From now on we fix $x_0\in \Gamma^\star$ and take $R_0>0$ as in Proposition \ref{prop:local_separation_property}. Denote by $\Omega_1,\Omega_2$ the two connected components of $B_{R_0}(x_0)\cap \{ U>0\}$ and by $u$ and $v$ the two functions amongst the components of the vector map $U$ that satisfy $B_{R_0}(x_0)\cap\{ u>0\} =\Omega_1$, $B_{R_0}(x_0)\cap \{v>0\}=\Omega_2$. Two situations may occur:
\begin{enumerate}
\item $u=u_i$ and $v=u_j$ in $B_{R_0}(x_0)$ for some $i\neq j$. In this case $u_k\equiv 0$ in $B_{R_0}(x_0)$ for $k\not\in \{i,j\}$ and $(u,v)=(u_i,u_j)\in \Geh(B_{R_0}(x_0))$.
\item $u_k\equiv 0$ for all $k\neq i$ for some $i$. In this case we take
$$
u(x)=\left\{
\begin{array}{cl}
u_i(x) & \text{ if } x\in \Omega_1\\
0 & \text{ if } x\in B_{R_0}(x_0)\setminus \Omega_1
\end{array}
\right.
\qquad
v(x)=\left\{
\begin{array}{cl}
u_i(x) & \text{ if } x\in \Omega_2\\
0 & \text{ if } x\in B_{R_0}(x_0)\setminus \Omega_2
\end{array}
\right.
$$
The next statement shows that $(u,v)\in \Geh(B_{R_0}(x_0))$ also in this situation.
\end{enumerate}

\begin{lemma}\label{lemma:ext_by_0}
Under the situation of case 2 described before we obtain $u,v\in H^1(B_{R_0}(x_0))$, $\nabla u=\nabla u_i \chi_{\Omega_1},\ \nabla v=\nabla u_i \chi_{\Omega_2}$ and the existence of non negative Radon measures $\lambda,\mu$ such that $\lambda_i=\lambda+\mu$ and
$$
\left\{
\begin{array}{l}
-\Delta u= f_i(x,u)-\lambda\\
-\Delta v=f_i(x,v)-\mu
\end{array}
\right.
\qquad
\text{ in } B_{R_0}(x_0).
$$
\end{lemma}
\begin{proof}
We prove the result for $u$ only. Take $\varphi\in \Deh(B_{R_0}(x_0))$ and consider a sequence $\varepsilon_n\rightarrow 0$ such that the sets $\{u>\varepsilon_n\}$ are regular (which exists by Sard's Theorem). We have
\begin{eqnarray*}
\int_{B_{R_0}(x_0)} u \nabla \varphi &=& \int_{\Omega_1} u_i \nabla \varphi = \lim_n \int_{\Omega_1\cap \{u_i>\varepsilon_n\}} u_i \nabla \varphi\\
&=& \lim_n \int_{\Omega_1\cap \{u_i>\varepsilon_n\}} -\nabla u_i \varphi + \lim_n \int_{\Omega_1 \cap \partial\{u_i>\varepsilon_n\}}u_i \varphi \nu\\
&=& \int_{\Omega_1}-\nabla u_i \varphi + \lim_n  \int_{\Omega_1\cap \{u>\varepsilon_n\}} \varepsilon_n \nabla \varphi = \int_{\Omega_1}-\nabla u_i \varphi
\end{eqnarray*}
and hence $\nabla u=\nabla u_i \chi_{\Omega_1}$. On the other hand the existence of the measure $\lambda$ comes from the fact that $\Delta u + f_i(x,u)\geq 0$ in $\Deh'(B_{R_0}(x_0))$: taking $\varphi\geq 0$,
\begin{eqnarray*}
\int_{B_{R_0}(x_0)} (u \Delta \varphi + f_i(x,u)\varphi) = \lim_n \int_{\Omega_1\cap \{u>\varepsilon_n\}} \left(u_i \Delta \varphi + f_i(x,u_i)\varphi \right) =\lim_n  \int_{\Omega_1\cap \partial\{u_i>\varepsilon_n\}}   \left( u_i \partial_\nu \varphi - \partial_\nu u_i \varphi \right).
\end{eqnarray*}
Now the result follows because
$$
\lim_n \int_{\Omega_1\cap \partial \{ u_i>\varepsilon_n\}} u_i \partial_\nu \varphi = \lim_n \int_{\Omega_1\cap \{u_i>\varepsilon_n\}}\varepsilon_n \Delta \varphi =0,
\qquad \text{and}\qquad \int_{\Omega_1\cap \partial \{u_i>\varepsilon_n\}} -\partial_\nu u_i \varphi \geq 0.
$$
\end{proof}

Hence in both cases the situation is the following: we have two non negative $H^1$--functions $u,v$ such that $u\cdot v=0$ in $B_{R_0}(x_0)$,  $B_{R_0}(x_0)\cap\{u>0\}=\Omega_1$, $B_{R_0}(x_0)\cap \{v>0\}=\Omega_2$, $B_{R_0}(x_0)\setminus \Gamma_U=\Omega_1\cup \Omega_2$, and there exist functions $f,g$ satisfying (G1) and nonnegative Radon measures $\lambda,\mu$ satisfying (G2) such that
$$
\left\{
\begin{array}{l}
-\Delta u= f(x,u)-\lambda\\
-\Delta v=g(x,v)-\mu
\end{array}
\right.
\qquad
\text{ in } B_{R_0}(x_0).
$$
Moreover assumption (G3) holds. To end this section we will prove that in fact $\lambda=\mu$ in $B_{R_0}(x_0)$, which will moreover imply that $\Gamma^\star \cap B_{R_0}(x_0)$ is a $C^{1,\alpha}$ hyper-surface.

\begin{lemma}[Reflection Principle] \label{lemma:reflection_principle} Let $\bar u,\bar v\in H^1_\text{loc}(\R^N)\cap C(\R^N)$ be two non zero and non negative functions in $\R^N$ such that $\bar u\cdot \bar  v = 0$ and
$$
\left\{
\begin{array}{l}
\Delta \bar u= \bar \lambda\\
\Delta \bar v = \bar \mu
\end{array}
\right. \qquad \qquad \text{ in } \R^N
$$
for some $\bar \lambda,\bar \mu\in \Mah_\text{loc}(\R^N)$, locally non negative Radon measures satisfying (G2). Suppose moreover that $\Gamma:=\Gamma_{(\bar u,\bar v)}=\partial\{\bar u>0\}=\partial \{\bar v>0\}$ is an hyper-plane and that $(G3)$ holds, that is
\begin{equation}\label{eq:derivative_of_E_at_limit}
\frac{d}{dr} E(x_0,(\bar u,\bar v),r)=\frac{2}{r^{N-2}}\int_{\partial B_r(x_0)}((\partial_\nu \bar u)^2+(\partial_\nu \bar v)^2)\, d\sigma \quad \text{ for every }x_0\in \R^N, r>0
\end{equation}
(where we  recall that $E(x_0,(\bar u,\bar v),r)=\frac{1}{r^{N-2}}\int_{B_r(x_0)}(|\nabla \bar u|^2+|\nabla \bar v|^2$) in this case).
Then for every Borel set $E\subseteq \R^N$ it holds
$$\bar \lambda (E)=\int_{E\cap \partial\{\bar u>0\}} -\partial_\nu \bar u\, d\sigma=\int_{E\cap \partial\{\bar v>0\}} - \partial_\nu \bar v\, d\sigma = \bar \mu(E)$$ and in particular $\Delta (\bar u-\bar v)=0$ in $\R^N$.
\end{lemma}

\begin{proof}
Suppose without loss of generality that $\Gamma=\R^{N-1}\times\{0\}$ and that $u\not\equiv 0$ in $\{x_N>0\}$, $v\not\equiv 0$ in $\{x_N< 0\}$. In this case we observe that $\bar u\in C^\infty(\{x_N\geq 0\})$, $\bar v\in C^\infty(\{x_N\leq 0\})$ and that our goal  is to check that
$$
\bar \lambda (E)=\int_{E\cap \Gamma} \partial_{e_N} \bar u\, d\sigma=\int_{E\cap \Gamma} - \partial_{e_N} \bar v\, d\sigma = \bar \mu(E),\\
$$
where $e_N$ is the vector $(0,\ldots,0,1)$. We divide the proof in two steps.

\noindent STEP 1. For every $E$ Borel set of $\R^N$ it holds
\begin{equation}\label{eq:expression_for_lambda_mu}
\bar \lambda(E)=\int_{E\cap \Gamma}\partial_{e_N} \bar u\,d\sigma \qquad \text{ and }\qquad  \bar \mu(E)=\int_{E\cap \Gamma}-\partial_{e_N} \bar v \, d\sigma.
\end{equation}
We present the proof of this claim only for $\bar \lambda$ - for $\bar \mu$ the computations are analogous. It suffices to prove that \eqref{eq:expression_for_lambda_mu} holds for every open ball $B_r(x_0)$. If $B_r(x_0)\cap \Gamma=\emptyset$ then $\bar \lambda(B_r(x_0))=0$ and equality holds. If on the other hand $B_r(x_0)\cap \Gamma\neq \emptyset$ then for any given $\delta>0$ take $\varphi_\delta$ to be a cut-off function such that $\varphi_\delta=1$ in $B_{r-\delta}(x_0)$, $\varphi_\delta=0$ in $\R^N\setminus B_r(x_0)$. We have
\begin{eqnarray*}
\int_{B_r(x_0)}\varphi_\delta\, d\bar\lambda &=& -\int_{B_r(x_0)}\nabla\bar u\cdot \nabla \varphi_\delta = -\int_{B_r(x_0)\cap \{\bar u>0\}}\nabla \bar u\cdot \nabla \varphi_\delta\\
&=&\int_{B_r(x_0)\cap \{\bar u>0\}}\Delta \bar u \varphi_\delta-\int_{B_r(x_0)\cap \Gamma}(\partial_{-e_N} \bar u )\varphi_\delta\, d\sigma\\
&=& \int_{B_r(x_0)\cap \Gamma}(\partial_{e_N} \bar u) \varphi_\delta\, d\sigma.
\end{eqnarray*}
Thus
$$
\bar \lambda(B_r(x_0))=\lim_{\delta\rightarrow 0}\int_{B_r(x_0)}\varphi_\delta\, d\bar \lambda = \int_{B_r(x_0)\cap \Gamma}\partial_{e_N} \bar u\, d\sigma.
$$

\noindent STEP 2. $\partial_{e_N} \bar u = -\partial_{e_N} \bar v$ in $\Gamma$.

By using the regularity of $\bar u, \bar v$ together with the fact that $\Gamma$ is an hyper-plane, we will compute the derivative of $E$ directly, and compare afterwards the result with expression \eqref{eq:derivative_of_E_at_limit}. Since $\bar u,\bar v\in H^1_\text{loc}(\R^N)$, then
\begin{eqnarray*}
\frac{d}{dr}E(x_0,(\bar u,\bar v),r)&=& \frac{2-N}{r^{N-1}}\int_{B_r(x_0)}(|\nabla \bar u|^2+|\nabla \bar v|^2)+\frac{1}{r^{N-2}}\int_{\partial B_r(x_0)}(|\nabla \bar u|^2+|\nabla \bar v|^2)\, d\sigma\\
&=& \frac{2-N}{r^{N-1}}\int_{B_r(x_0)\cap \{\bar u>0\}}|\nabla \bar u|^2+\frac{1}{r^{N-2}}\int_{\partial B_r(x_0)\cap \{\bar u>0\}}|\nabla \bar u|^2\, d\sigma+\\
&&+\frac{2-N}{r^{N-1}}\int_{B_r(x_0)\cap \{\bar v>0\}}|\nabla \bar v|^2+\frac{1}{r^{N-2}}\int_{\partial B_r(x_0)\cap \{\bar v>0\}}|\nabla \bar v|^2\, d\sigma.
\end{eqnarray*}
In order to rewrite the integrals on $\partial B_r(x_0)$, we use the following Rellich--type identity
\begin{equation}\label{eq:Pohoazaev}
\text{div}\left( (x-x_0)|\nabla \bar u|^2-2\langle x-x_0,\nabla \bar u\rangle \nabla \bar u \right)=(N-2)|\nabla \bar u|^2-2 \langle x-x_0,\nabla \bar u \rangle\Delta \bar u
\end{equation}
in $B_r(x_0)\cap \{\bar u>0\}$ (recall that $\bar u$ is smooth in this set). By the fact that $\Delta \bar u=0$ in the latter set and that $\nabla \bar u=(\partial_{e_N} \bar u)e_N$ on $\partial \{\bar u>0\}=\Gamma$, we have
$$
\int_{\partial B_r(x_0)\cap\{\bar u>0\}}|\nabla \bar u|^2=2\int_{\partial B_r(x_0)\cap\{\bar u>0\}}(\partial_\nu \bar u)^2 - \frac{1}{r}\int_{B_r(x_0)\cap \Gamma}(\partial_{e_N} \bar u)^2 \langle e_N,x-x_0\rangle+\frac{N-2}{r}\int_{B_r(x_0)\cap\{\bar u>0\}}|\nabla \bar u|^2
$$
and analogously
$$\int_{\partial B_r(x_0)\cap\{\bar v>0\}}|\nabla \bar v|^2=2\int_{\partial B_r(x_0)\cap\{\bar v>0\}}(\partial_\nu \bar v)^2 + \frac{1}{r}\int_{B_r(x_0)\cap \Gamma}(\partial_{e_N} \bar v)^2 \langle e_N,x-x_0\rangle+\frac{N-2}{r}\int_{B_r(x_0)\cap\{\bar v>0\}}|\nabla \bar v|^2.$$
Thus
$$
\frac{d}{dr}E(x_0,(\bar u,\bar v),r)=\frac{2}{r^{N-2}}\int_{\partial B_r(x_0)}( (\partial_\nu \bar u)^2+(\partial_\nu \bar v)^2   )\, d\sigma + \frac{1}{r^{N-1}}\int_{B_r(x_0)\cap \Gamma}[(\partial_{e_N} \bar v)^2-(\partial_{e_N} \bar u)^2]\langle e_N,x-x_0  \rangle
$$
which, comparing with \eqref{eq:derivative_of_E_at_limit}, yields that
$$
\int_{B_r(x_0)\cap \Gamma} [(\partial_{e_N} \bar v)^2-(\partial_{e_N} \bar u)^2]\langle e_N,x-x_0  \rangle=0 \qquad \text{ for every }x_0\in\R^N,\ r>0,
$$
and therefore $(\partial_{e_N} \bar v)^2=(\partial_{e_N} \bar u)^2$ on $\Gamma$. Finally we just have to observe that
$|\partial_{e_N \bar u}|=\partial_{e_N} \bar u$ and $|\partial_{e_N}\bar v|=-\partial_{e_N}\bar v$.
\end{proof}

\begin{teo}\label{teo:global_equation}
With the previous notations we have $\lambda(E)=\mu(E)$ for every $E$ Borel set of $B_{R_0}(x_0)$, and in particular
\begin{equation}\label{eq:equation_with_2_components}
 -\Delta (u-v)=f(x,u)-g(x,v) \quad \text{ in } B_{R_0}(x_0).
 \end{equation}
\end{teo}
\begin{proof}
We claim that
$$\lim_{r\rightarrow 0}\frac{\lambda(\bar B_r(y))}{\mu(\bar B_r(y))}=1 \qquad \text{ for every }y\in \Gamma_U\cap B_{R_0}(x_0).$$
Fix $y\in \Gamma_U\cap B_{R_0}(x_0)$ and consider any arbitrary sequence $r_k\downarrow 0$. If we define $u_k(x)=u(y+r_k x)/\rho_k$, $v_k(x)=v(y+r_k x )/\rho_k$ as a usual blowup sequence at a point $y$, and consider $\lambda_k,\mu_k$ to be the associated rescaled measures, then Theorem \ref{teo:blow_up_convergence} yields the existence of a pair of functions $(\bar u,\bar v)\in \Geh_\text{loc}(\R^N)$ and measures $(\bar \lambda, \bar \mu)$ such that
\begin{eqnarray*}
u_k \rightarrow \bar u, & \quad v_k\rightarrow \bar v &\qquad \text{ in } H^1_\text{loc}\cap C^{0,\alpha}_\text{loc}\\
\lambda_k\rightharpoonup \bar \lambda,&\quad \mu_k\rightharpoonup \bar \mu & \qquad \text{ in the measure sense,}
\end{eqnarray*}
and $\Delta \bar u=\bar \lambda$, $\Delta \bar v=\bar \mu$ in $\R^N$.
Property (P) implies that $\Gamma_{(\bar u,\bar v)}$ is a hyper-plane passing through the origin. From this fact, the uniform convergence of $u_k,v_k$ to $\bar u,\bar v$, and the second statement of Proposition \ref{prop:local_separation_property}, we deduce also that $\bar u,\bar v\neq 0$. Thus we can apply Lemma \ref{lemma:reflection_principle} to the functions $\bar u,\bar v$, which provides
$$\bar \lambda (E)=\int_{E\cap \partial\{\bar u>0\}} -\partial_\nu \bar u\, d\sigma=\int_{E\cap \partial\{\bar v>0\}} - \partial_\nu \bar v\, d\sigma = \bar \mu(E)$$
for every Borel set $E$ of $\R^N$. In particular $\bar \lambda(\bar B_1(0))=\bar \mu(\bar B_1(0))\neq 0$ and $\bar \lambda(\partial B_1(0))=\bar \mu(\partial B_1(0))=0$, thus
$$\lambda_k(\bar B_1(0))\rightarrow \bar \lambda(\bar B_1(0)),\qquad \mu_k(\bar B_1(0))\rightarrow \mu(\bar B_1(0))$$
(see for instance \cite[\S 1.6--Theorem 1]{evans_mt}) and
$$1=\frac{\bar \lambda(\bar B_1(0))}{\bar \mu(\bar B_1(0))}=\lim_k\frac{\lambda_k(\bar B_1(0))}{\mu_k(\bar B_1(0))}=\lim_k \frac{\lambda(\bar B_{r_k}(y))}{\mu(\bar B_{r_k}(y))},$$as claimed.

Therefore $D_{\mu}\lambda(y)=1$ for $\mu$--a.e. $y\in B_{R_0}(x_0)$ and $D_{\lambda}\mu (y)=1$ for $\lambda$--a.e. $y\in B_{R_0}(x_0)$ (recall that both $\lambda$ and $\mu$ are supported on $\Gamma$), and hence the Radon-Nikodym Decomposition Theorem (see for instance \cite[\S 1.6 - Theorem 3]{evans_mt}) yields that for every Borel set $E\subseteq B_{R_0}(x_0)$
\begin{eqnarray*}
\lambda(E)=\lambda_s(E)+\mu(E)\geq \mu(E)\\
\mu(E)=\mu_s(E)+\lambda(E)\geq \lambda(E)
\end{eqnarray*}
(where $\lambda_s\geq 0$ represents the singular part of $\lambda$ with respect to $\mu$ and $\mu_s\geq 0$ represents the singular part of $\mu$ with respect to $\lambda$). Hence $\lambda(E)=\mu(E)$, which concludes the proof of the theorem.
\end{proof}

With the following result we end the proof of Theorem \ref{teo:Sigma_U_smooth}.

\begin{coro}\label{coro:gradient_non_zero}
Under the previous notations, $u-v\in C^{1,\alpha}(B_{R_0}(x_0))$ for every $0<\alpha<1$, and $$\nabla (u-v)(x_0)\neq 0.$$
\end{coro}
\begin{proof}
Since $w=u-v$ solves $-\Delta w=f(x,w^+)-g(x,w^-)$ and $f(x,w^+)-g(x,w^-)\in L^\infty(B)$, then standard elliptic regularity yields $w\in C^{1,\alpha}(B_{R_0}(x_0))$ for all $0<\alpha<1$. Now if we consider a blowup sequence centered at $x_0$, namely $w_k(x):=(u(x_0+t_k x)-v(x_0+t_k x))/\rho_k$ then
$$
\begin{array}{c}
w_k\rightarrow \bar w:=\bar u-\bar v \qquad \text{ in } H^1_\text{loc}\cap C^{0,\alpha}_\text{loc}(B_2(0))\\
-\Delta w_k=f_k(x,u_k)-g_k(x,v_k)\rightarrow 0 \quad \text{ in } L^\infty(B_2(0))\\
\Delta \bar w=0 \qquad \text{ in } B_2(0)
\end{array}
$$
and hence
$$\|w_k-\bar w\|_{C^{1,\alpha}(B_1(0))}\leq C( \|w_k-\bar w\|_{L^\infty(B_2(0))}+\|f_k(x,u_k)-g_k(x,v_k))\|_{L^\infty(B_2(0))})\rightarrow 0.$$
Since (by Corollary \ref{coro:U_homogeneous}) $\bar w$ is a homogeneous function of degree one, then $\nabla \bar w(0)\neq 0$ and thus also $\nabla w_k(0)=r_k\nabla w (x_0)/\rho_k\neq 0$ for large $k$.
\end{proof}

\begin{proof}[Proof of Theorem \ref{teo:Sigma_U_smooth}.]
Corollary \ref{coro:gradient_non_zero} implies by the Implicit Function Theorem that $\Gamma^\star$ is indeed a $C^{1,\alpha}$ hyper-surface. Furthermore, equation \eqref{eq:equation_with_2_components} implies the reflection principle \eqref{eq:reflection_principle_section5}.
\end{proof}

\begin{rem}\label{rem:vector_case} We consider here the case when the functions $u_i$ may be vector valued. In this case, we apply the previous results to the positive and negative parts of each amongsts their scalar components. The reflection Lemma \ref{lemma:reflection_principle} still holds and gives equality of  the total variations $\Vert\overline{ \lambda}\Vert(E)=\Vert\overline{\mu}\Vert(E)$ . Consequently, also Theorem \ref {teo:global_equation} holds for the total variations of the measures $\lambda$ and $\mu$. In contrast,  Corollary \ref{coro:gradient_non_zero} in no longer available for the case of vector valued components $u_i$. In order to complete the proof, we have to exploit  the iterative argument introduced in \cite{CL} in order to improve the flatness of the free boundary. The proof makes use of the boundary regularity theory by Jerison and Kenig and Kenig and Toro in non tangentially accessible and Reifenberg flat domains (see \cite{JK,KT}) and provides $C^{1,\alpha}$ regularity of the regular part of the nodal set. 
\end{rem}
\section{Proof of the main result in any dimension $N\geq 2$: iteration argument.}\label{sec:iteration_argument}

Given $N\geq 2$, by taking in consideration Theorems \ref{teo:Hausdorff_dim_estimates} and \ref{teo:Sigma_U_smooth} as well as Remark \ref{rem:what_do_we_need_to_have_hausdorff_estimates}, we deduce that in order to prove our main result (Theorem \ref{teo:main_result}) it is enough to check the following.

\begin{lemma}\label{lemma:claim} Let $N\geq 2$. Given $\bar U=r^\alpha G(\theta)\in \Geh_\text{loc}(\R^N)$  such that $\Delta \bar U=0$ in $\{\bar U>0\}$, then either $\alpha=1$ or $\alpha\geq 1+\delta_N$ for some universal constant $\delta_N$ depending only on the dimension. Moreover if $\alpha=1$ then $\Gamma_{\bar U}$ is an hyper-plane.
\end{lemma}
In fact, assuming for the moment that Lemma \ref{lemma:claim} holds:

\begin{proof}[Proof of Theorem \ref{teo:main_result}] Fix $N\geq 2$, $\Omega\subseteq \R^N$ and let $U\in \Geh(\Omega)$. By Theorem \ref{teo:Hausdorff_dim_estimates}-1 we have $\Hh_\text{dim}(\Gamma_U)\leq N-1$. Next, for each $x_0\in \Omega$, take a blowup sequence $U_k(x)=U(x_0+t_k x)/\rho_k$. Theorem \ref{teo:blow_up_convergence} and Corollary \ref{coro:U_homogeneous} (case 1) together imply the existence of a blowup limit $\bar U=r^\alpha G(\theta)\in \Geh_\text{loc}(\R^N)$ such that $\Delta \bar U=0$ in $\{\bar U>0\}$, and $\alpha=N(x_0, U,0^+)$. Thus we can apply Lemma \ref{lemma:claim} which allows us to deduce that either $N(x_0, U,0^+)=1$ or $N(x_0, U,0^+)\geq 1+\delta_N$, for some universal constant $\delta_N>0$. In this way, being $S_U,\Sigma_U$ the sets defined in Definition \ref{def:singular_regular_set}, we obtain (by repeating exactly the proofs of Corollary \ref{coro:S_U_closed} and Theorem \ref{teo:Hausdorff_dim_estimates}-2) that $S_U$ is closed, $\Sigma_U$ is relatively open in $\Gamma_U$, and that $\Hh_\text{dim}(\Sigma_U)\leq N-2$. Finally, Corollary \ref{coro:U_homogeneous} (case 2) and Lemma \ref{lemma:claim} imply that $\Gamma^\star:=\Sigma_U$ satisfies condition (P) in Theorem \ref{teo:Sigma_U_smooth}, which shows that $\Sigma_U$ is a $C^{1,\alpha}$ hyper-surface and that \eqref{eq:reflection_principle} holds.

Furthermore, in dimension N=2, we know from Theorem \ref{teo:Hausdorff_dim_estimates}-2 that $S_U$ is locally a finite set. For each $y_0\in S_U$ take a small radius such that $S_U\cap B_r(y_0)=\{y_0\}$. Since \eqref{eq:reflection_principle} holds, we can apply the same reasoning of Theorem 9.6 in \cite{CTV3} to the ball $B_r(y_0)$, proving this way that $\Sigma_U\cap B_r(y_0)$ is a finite collection of curves meeting with equal angles at $y_0$, which is a singular point.
\end{proof}

The remainder of this section is devoted to the proof of Lemma \ref{lemma:claim}. Its proof follows by induction in the dimension $N$. For $N=2$ the statement holds by Proposition \ref{prop:N=1_or_N_geq_1+delta} and Remark \ref{rem:N=2_N=1_nodalsetishyper-plane}. Suppose now that the claim holds in dimension $N-1$ and take $\bar U=r^\alpha G(\theta)\in \Geh_\text{loc}(\R^N)$ such that $\Delta \bar U=0$ in $\{\bar U>0\}$. We first treat the case in which the positive set has three or more connected components. In three dimensions the exact value of $\delta_N$ has been proven to be $1/2$ in \cite{HHOT2}.

\begin{lemma}
If $\{G>0\}$ has at least three connected components then there exists an universal constant $\bar \delta_N>0$ such that $\alpha\geq 1+\bar \delta_N$.
\end{lemma}
\begin{proof}
We argue exactly as in the first part of the proof of Proposition \ref{prop:N=1_or_N_geq_1+delta} (from which we also recall the definition of $E(\theta)$). Note that for every connected component $A\subseteq \{g_i>0\}\subset S^{N-1}$ it holds
$$
-\Delta_{S^{N-1}}g_i=\lambda g_i \quad \text { in } A,\ \qquad \text{ with } \lambda=\alpha(\alpha+N-2) \text{ and }\lambda=\lambda_1(A).
$$
 At least one of the connected components, say $C$, must satisfy $\Hh^{N-1}(C)\leq \Hh^{N-1}(S^{N-1})/3$, and hence $\lambda=\lambda_1(C)\geq \lambda_1(E(\pi/2))$. Moreover it is well know that $\lambda_1(E(\pi/2))=N-1$. This implies the existence of $\gamma>0$ such that $\lambda_1(E(\pi/3))=N-1+\gamma$, and thus $\alpha=\sqrt{\left(\frac{N-2}{2}\right)^2+\lambda}-\frac{N-2}{2}\geq 1+\bar \delta_N$ for some $\bar \delta_N>0$.
\end{proof}

From now on we suppose that $\{G>0\}$ has at most two connected components. In order to prove Lemma \ref{lemma:claim} the next step is to study the local behaviour of the function $\bar U$ at its non zero nodal points $y_0\in \Gamma_{\bar U}\setminus \{0\}$. This study is accomplished by performing a new blowup analysis. Because $\bar U$ is homogeneous it suffices to take blowup sequences centered at $y_0\in \Gamma_{\bar U}\cap S^{N-1}=\Gamma_G$.

Fix $y_0\in \Gamma_{\bar U}\cap S^{N-1}$ and consider $V_k(x):=\bar U(y_0+t_k x)/\rho_k$ for some $t_k \downarrow 0$ and $\rho_k=\|\bar U(y_0+t_k \cdot )\|_{L^2(\partial B_1(0))}$. Theorem \ref{teo:blow_up_convergence} and Corollary \ref{coro:U_homogeneous} provide the existence of a blowup limit $\bar V=r^\gamma H(\theta)\in \Geh_\text{loc}(\R^N)$, with $\gamma=N(y_0,\bar U,0^+)$. By the homogeneity of $\bar U$ we are able to prove that $\bar V$ actually depends only on $N-1$ variables.

\begin{lemma}\label{lemma:V_does_not_depend_on_x_N}
It holds $\bar V(x+\lambda y_0)=\bar V(x)$ for every $\lambda>0$, $x\in \R^N$.
\end{lemma}
\begin{proof}
Fix $x\in \R^N$ and $\lambda>0$. Recall that $V_k\rightarrow \bar V$ in $C^{0,\alpha}_\text{loc}(\R^N)$, which in particular implies pointwise convergence. Hence in particular $V_k(x)\rightarrow \bar V(x)$ and $V_k(x+\lambda y_0)\rightarrow \bar V(x+\lambda y_0)$. In order to prove the lemma it is enough to check that $\lim_k (V_k(x+t y_0)-V_k(x))=0$. From the homogeneity of $\bar U$ one obtains
\begin{eqnarray*}
V_k(x+\lambda y_0) &=& \frac{1}{\rho_k}\bar U(y_0+t_k(x+\lambda y_0))=\frac{1}{\rho_k}\bar U ((1+\lambda t_k )y_0 + t_k x)\\
&=& \frac{(1+\lambda t_k)^\alpha}{\rho_k}\bar U\left(y_0+\frac{t_k}{1+\lambda t_k}x\right)=(1+\lambda t_k)^\alpha V_k\left(\frac{x}{1+t_k\lambda}\right).
\end{eqnarray*}
Take a compact set $K$ containing $x$ and $x/(1+\lambda t_k)$ for large $n$. There exists a constant $C=C(K)$ such that
\begin{eqnarray*}
|V_k(x+\lambda y_0)-V_k(x)|&=&\left|(1+\lambda t_k)^\alpha V_k\left(\frac{x}{1+t_k \lambda}\right)-V_k(x)\right|\\
&\leq& \left|(1+\lambda t_k)^\alpha V_k\left(\frac{x}{1+\lambda t_k}\right)-V_k\left(\frac{x}{1+\lambda t_k}\right)\right|+\left|V_k\left(\frac{x}{1+\lambda t_k }\right)-V_k(x)\right|\\
&\leq& C |(1+\lambda t_k)^\alpha-1| + C\left|\frac{1}{1+\lambda t_k}-1\right|^\alpha |x|^\alpha \rightarrow 0
\end{eqnarray*}
\end{proof}
Next we use the induction hypothesis in order to prove a jump condition of the possible values of $\gamma=N(y_0,\bar U,0^+)$.

\begin{lemma}\label{lemma:gamma=1_or_gamma_geq_1+delta}
With the previous notations either $\gamma\geq 1$ or $\gamma\geq 1+\delta_{N-1}$. Furthermore if $\gamma=1$ then $\Gamma_{\bar V}$ is a hyper-plane.
\end{lemma}
\begin{proof}
Up to a rotation we can suppose that $y_0=(0,\ldots,0,1)$. Hence by Lemma \ref{lemma:V_does_not_depend_on_x_N} $\bar V(x)=V(x_1,\ldots,x_{N-1})=|(x_1,\ldots,x_{N-1})|^{\gamma}H\left(\frac{(x_1,\ldots,x_{N-1})}{|(x_1,\ldots,x_{N-1})|}\right)$, $\Delta_{\R^{N-1}} \bar V=0$ in $\{\bar V>0\}$ and $\bar V_{|\R^{N-1}\times \{0\}}\in \Geh_\text{loc}(\R^{N-1})$. Hence by the induction hypothesis either $\gamma=1$ or $\gamma\geq 1+\delta_{N-1}$. Moreover if $\gamma=1$ then $\Gamma_{\bar V}\cap \left(\R^{N-1}\times \{0\}\right)$ is an $(N-2)$--dimensional subspace of $\R^{N-1}$ and hence $\Gamma_{\bar V}$ is an hyper-plane in $\R^N$.
\end{proof}

The previous result shows that given $y_0\in \Gamma_{\bar U}\cap S^{N-1}$ then either $N(y_0,\bar U,0^+)=1$ or $N(y_0,\bar U,0^+)\geq 1+\delta_{N-1}$.

\begin{lemma}
Suppose there exists $y_0\in \Gamma_{\bar U}\cap S^{N-1}$ such that $N(y_0,\bar U,0^+)\geq 1+\delta_{N-1}$. Then $\alpha=N(0,\bar U,0^+)\geq 1+\delta_{N-1}$.
\end{lemma}
\begin{proof}
Take, for every $t>0$, the rescaled function $\bar U_{0,t}(x):= \bar U(tx) =t^\alpha \bar U(x)$. By taking into account identity \eqref{eq:almgren_identities2} we obtain that for every $r>0$
$$
N(y_0,\bar U,r)= N(y_0, t^\alpha \bar U,r)=N(y_0,\bar U_{0,t},r)=N(t y_0,\bar U,t r).
$$
Therefore $N(t y_0,\bar U,0^+)=N(y_0,\bar U,0^+)\geq 1+\delta_{N-1}$ and the conclusion of the lemma follows from the upper semi-continuity of the function $y\mapsto N(y,\bar U,0^+)$ (Corollary \ref{coro:N_upper_semi_continuous}).
\end{proof}

From now on we suppose that the set $\{G>0\}$ has at most two connected components and that $N(y_0,\bar U, 0^+)=1$ for every $y_0\in \Gamma_G$. Let us prove that $\alpha\in \N$ and that if $\alpha=1$ then $\Gamma_{\bar U}$ is an hyper-plane (in the remaining cases we have shown that $\alpha\geq 1+\min\{\bar \delta_N,\delta_{N-1}\})$.

Observe that the second conclusion in Lemma \ref{lemma:gamma=1_or_gamma_geq_1+delta} shows that property (P) holds at every point $y_0\in \Sigma_{\bar U}\cap S^{N-1}=\Gamma_{\bar U}\cap S^{N-1}$. Hence Theorem \ref{teo:Sigma_U_smooth} yields that $\nabla \bar U(y_0)\neq 0$ whenever $y_0\in \Gamma_{\bar U}\cap S^{N-1}$, and in particular $\nabla_\theta \bar U(y_0)\neq 0$ since $\bar U$ is a homogeneous function and $\bar U(y_0)=0$. In this way we conclude that the set $\Gamma_{\bar U}\cap S^{N-1}$ is a compact $(N-2)$-- dimensional sub-manifold of $S^{N-1}$ without boundary, and by a generalization of the Jordan Curve Theorem we conclude that in fact $S^{N-1}\setminus \Gamma_{\bar U}$ - as well as $\R^N\setminus \Gamma_{\bar U}$ - is made of two connected components.

Denote by $\Omega_1,\Omega_2$ the two connected components of $\R^N\setminus \Gamma_{\bar U}$ and respectively by $u,v$ the non trivial components of $\bar U$ in the latter sets. Once again by Theorem \ref{teo:Sigma_U_smooth} we obtain that $\nabla u=-\nabla v$ on $\Gamma_{\bar U}\setminus\{0\}$ and hence $\Delta (u-v)=0$ in $\R^N$, and $(u,v)=r^\alpha G(\theta)$. Thus $\alpha\in \N$ and if $\alpha=1$ then $\nabla (u-v)(0)\neq 0$ and $\Gamma_{\bar U}$ is a hyper-surface.

In conclusion we have proved the conclusion of Lemma \ref{lemma:claim} in any dimension $N$, more precisely we have shown that either $\alpha\geq 1+\min\{2,\bar \delta_N,\delta_{N-1}\}$ or else $\alpha=1$ and $\Gamma_{\bar U}$ is a hyper-plane.


\section{Elliptic operators with variable coefficients}\label{sec:riemannian}
Theorem \ref{teo:main_result} extends to segregated configurations associated with systems of semilinear elliptic equations on Riemannian manifolds, under an appropriate version of the weak reflection law. In order to clarify the geometrical meaning of the weak reflection principle and to understand which version of assumption (G3) makes possible such an extension, we start  with a system of semilinear equations involving the Laplace-Beltrami operator on a Riemannian manifold $M$:
\[-\Delta _M  u_i=f(x,u_i)-\mu_i\;.\] 
We assume that (G1) and (G2) hold and we define the  ``energy'' $\tilde E$ as  
\begin{equation}\nonumber
          \tilde E(r)=\tilde E(x_0,U,r)=\frac{1}{r^{N-2}}\int_{B_r(x_0)}|\nabla_M U|^2 dV_M,
          \end{equation}
{\it  where $B_r(x_0)$ is the geodesic ball of radius $r$}. Let us  choose normal coordinates ${\tilde x}^i$ centered at $x_0$. By Gauss Lemma we know that, denoting by $\rho=\sum_i ({\tilde x}^i)^2$ and $\theta^i$ the radial and angular coordinates,  it holds
 \[g=d\rho^2+\rho^2\sum_{i,j}b_{ij}(\rho,\theta)d\theta^id\theta^j.\]
 
 Notice that the variation with respect to the euclidean metric is purely tangential. Moreover the Christoffel symbols vanish at the origin. In such coordinates, denoting, as usual, $\tilde g_{ij}=g(\partial_i,\partial_j)$ the coefficients of the metric with respect to the normal coordinates, we require that $\tilde E$ satisfies the differential equation:
   \begin{multline}\label{G3_manifolds}
   \frac{d}{dr}\tilde E(x_0,U,r)= \frac{2}{r^{N-2}}\int_{\partial B_r(x_0)}  (\partial_\rho U )^2\, d\sigma_M \\
   +\frac{2}{r^{N-1}}\int_{B_r(x_0)} \rho\sum_{i}\left[f_i(x,u_i)\partial_\rho u_i + \dfrac{1}{\sqrt {\tilde g}}\sum_{k,j}\partial _\rho\left(\sqrt{\tilde g}\, \tilde g^{kj} \right)\partial_k u_i\partial_j u_i\right]\, d V_M.   \end{multline}

 Here $\tilde g=|\det(\tilde g_{kj})|$ and $(\tilde g^{kj})$ is the inverse  of the matrix $(\tilde g_{kj})$. As shown in \cite{GL}, this identity is satisfied also in the case of Lipschitz metrics, by any solution $u$ of the semilinear equation
 \[-\Delta_M u = f(x,u).\]

Similarily, when dealing with varying coefficients elliptic equations $Lu_i= -\textrm{div} (A(x)\nabla u_i)=f(x,u_i)-\mu_i$, we can associate with the coefficient matrix $A$ a metric $g$ in such a way that $A=\sqrt g \left(g^{ij}\right)^{ij}$. We denote by $M$ the associated Riemannian manifold. Next, {\it denoting by $B_r(x_0,r)$ the geodesic balls with respect to such metric}, according with the previous discussion, we  define  the energy as
$$\tilde E(r)=\tilde E(x_0,U,r)=\frac{1}{r^{N-2}}\int_{B_r(x_0)}|\nabla_M U|^2 dV_M=\frac{1}{r^{N-2}}\int_{B_r(x_0)}\langle A(x) \nabla U(x), \nabla U(x)\rangle dx\;.$$
 
Now let us consider again the normal coordinates $(\tilde x^i)^i$ for the metric $g$ and let {\it $\tilde g_{ij}$ be the coefficients of the metric in such normal coordinates}.  The new coefficient matrix $\tilde A=\sqrt{\tilde g}\tilde g^{ij}$ has the radial direction $x-x_0$ as  an eigenvector corresponding to the eigenvalue $\sqrt{\tilde g}$. Obviously in normal coordinates the geodesic ball centered at $x_0$ and the euclidean one coincide and $d\sigma_M=\sqrt{\tilde g}\,ds$ (here $ds$ denotes the standard euclidean metric on the sphere). If $x=\Phi(\tilde x)$, we denote $\tilde U=U\circ \Phi$ and we need that the energy  $\tilde E$ satisfies the differential equation
\begin{multline}\label{G3_variable}
   \frac{d}{dr}\tilde E(x_0,U,r)= \frac{2}{r^{N-2}}\int_{\partial B_r(x_0)}  \sqrt{\tilde g} \,(\partial_\rho \tilde U )^2\, ds\\
   +\frac{2}{r^{N-1}}\int_{B_r(x_0)} \rho\sum_{i}\left[\sqrt{\tilde g}\, f_i(\Phi(\tilde x),\tilde u_i)\partial_\rho \tilde u_i + \sum_{k,j}\partial _\rho\left(\sqrt{\tilde g}\, \tilde g^{kj} \right)\partial_k \tilde u_i\partial_j \tilde u_i\right]\, d\tilde x \end{multline}
   Finally, it is convenient to scale further $\tilde u_i\mapsto \tilde u_i /{\root 4\of {\tilde g}} $, as we prefer to get rid of the Jacobian in the first term of the above identity. The coefficient matrix for the corresponding elliptic equations now is $\tilde A/{\sqrt{\tilde g}}$ and has the radial direction as an eigevector corresponding the the eigenvalue one. 

 The next step is to use any of the two equations \eqref{G3_manifolds} and \eqref{G3_variable} in order to prove Almgen's monotonicity formula. As pointed out in \cite{GL},  this can be done easily once we observe that the last term in the expression of the derivative is actually bounded by a constant times the energy itself (this happens, in general, for Lipschitz  metrics).  The rest of the proof remains unchanged.

\begin{rem}\label{rem:global_Lipschitz}
Having  learned how to extend Theorem \ref{teo:main_result} to the case of variable coefficients operators, we can now examine to which extent there holds regularity of the nodal set up to the boundary, under the regularity assumption $\partial\Omega\in\mathcal C^2$. To do this, we first need to extend the components $u_i$ by reflection through the bondary, exploiting a nonlinear reflection field $\Phi:\widehat \Omega \to \Omega$. Here  $\Phi$ is a $\mathcal C^2$ extension of the identity over $\overline{\Omega}$ in an open neighbourhood $\widehat \Omega$.   We associate with this extension field the metric $g$ having coefficient matrix $d\Phi\cdot d\Phi^*$ with respect to the euclidean coordinates. Then, the compositions $u_i\circ\Phi$ satisfy a system of semilinear elliptic equations involving the Laplace-Beltrami operator with respect to such a metric.  In order to apply Theorem \ref{teo:main_result},  we require  that \eqref{G3_manifolds} holds. A word of caution must be entered at this point: \eqref{G3_manifolds} is expressed in terms of  the coefficients of the metric with respect to the normal coordinates associated with the metric. Hence, in order to check its validity, a further change of coordinates is needed. Fortunately, we never check it directly in the applications, for we rather argue indirectly, passing to the limit in  the approximating procedure.

\end{rem}


\section{Applications.}\label{sec:applications}

In this last section we provide two applications of the previously developed theory. In both cases we prove that the functions in consideration belong to the class $\Geh(\Omega)$, and hence Theorem \ref{teo:main_result} applies.

\subsection{Asymptotic limits of a system of Gross-Pitaevskii equations}\label{subsec:application_BEC}

Consider the following system of nonlinear Schr\"odinger equations
\begin{equation}\label{eq:BEC}
\left\{\begin{array}{l}
-\Delta u_i +\lambda_i u_i = \omega_i u_i^3 - \beta u_i \sum_{j\neq i}\beta_{ij} u_j^2\\
u_i\in H^1_0(\Omega),\ u_i>0 \text{ in } \Omega.
\end{array}\right.
\qquad i=1,\ldots,h,
\end{equation}
in a smooth bounded domain $\Omega\subset \R^N$, $N=2,3$. Such a system arises in the theory of Bose-Einstein condensation (we refer to \cite{CLLL} and references therein). Here we consider $\beta_{ij}=\beta_{ji}\neq 0$ (which gives a variational structure to the problem) and take $\lambda_i,w_i\in \R$ and $\beta\in (0,+\infty)$ large. The existence of solutions for $\beta$ large is still an open problem for some choices of $\lambda_i,w_i$; for recent works on the subject see for instance \cite{DWW,NorisRamos,NTTV2} and references therein.

One of the many interesting questions about system \eqref{eq:BEC} is the asymptotic study of its solutions as $\beta\rightarrow +\infty$ (which represents an increasing of the interspecies scattering length), namely the regularity study of the limiting profiles. In the paper \cite{uniform_holder}, in collaboration with Noris and Verzini, we have proved $C^{0,\alpha}$-- bounds (for all $0<\alpha<1$) for any given $L^{\infty}$--bounded family of solutions $U_\beta=(u_{1,\beta},\ldots,u_{h,\beta})$ of \eqref{eq:BEC}. Moreover the possible limit configurations $U=\lim_{\beta\rightarrow +\infty} U_\beta$ are proved to be Lipschitz continuous. The mentioned paper contains the proof of the following fact.

\begin{teo}\label{teo:BEC_G(Omega)}
Let $U$ be a limit as $\beta\rightarrow +\infty$ of a family $\{U_\beta\}$ of $L^\infty$--bounded solutions of \eqref{eq:BEC}. Then $U\in \Geh(\Omega)$.
\end{teo}
\begin{proof}
For $f_i(x,s)=f_i(s)=\omega_i s^3-\lambda_i s$, Theorem 1.2 in \cite{uniform_holder} implies that $U$ satisfies each property in the definition of the class $\Geh(\Omega)$ except for (G3). The fact that this latter property is also satisfied is the content of the first part of the proof of Proposition 4.1 in \cite{uniform_holder}. The procedure is the following: defining an approximated ``energy'' associated with system \eqref{eq:BEC} - which has a variational structure-,
$$
E_\beta(r)=\frac{1}{r^{N-2}}\int_{B_r(x_0)}\left(|\nabla U_\beta|^2-\langle F(U_\beta),U_\beta \rangle \right)+\int_{B_r(x_0)}2\beta \sum_{i<j}u_{i,\beta}^2 u_{j,\beta}^2
$$
by a direct calculation it holds
\begin{multline*}
E_\beta'(r) = \frac{2}{r^{N-2}}\int_{\partial B_r(x_0)}\left( \partial_\nu U_\beta \right)^2\, d\sigma + \frac{2}{r^{N-1}}\int_{B_r(x_0)}\sum_i f_i(u_{i,\beta})\langle \nabla u_{i,\beta},x-x_0\rangle+\\
+ \frac{1}{r^{N-1}}\int_{B_r(x_0)}(N-2)\langle F(U_\beta),U_\beta \rangle - \frac{1}{r^{N-2}}\int_{\partial B_r(x_0)}\langle F(U_\beta),U_\beta\rangle\, d\sigma+\\
+\frac{4-N}{r^{N-1}}\int_{B_r(x_0)}\beta \sum_{i<j}u_{i,\beta}^2 u_{j,\beta}^2 + \int_{\partial B_r(x_0)} \beta \sum_{i<j}u_{i,\beta}^2 u_{j,\beta}^2\,d\sigma.
\end{multline*}
By \cite[Theorem 1.2]{uniform_holder} we obtain strong convergence $U_\beta\rightarrow U$ in $H^1\cap C^{0,\alpha}(\Omega)$ for every $0<\alpha<1$, and $\int_{\Omega} \beta \sum_{i<j}u_{i,\beta}^2 u_{j,\beta}^2\rightarrow 0$. Hence, as $\beta\rightarrow +\infty$, we prove that $U$ satisfies (G3) exactly in the same way we did at the end of the proof of Theorem \ref{teo:blow_up_convergence}.
\end{proof}

Hence Theorems  \ref{teo:main_result} and \ref{teo:BEC_G(Omega)} provide a new regularity result for asymptotic limits of general families of uniformly bounded excited state solutions of \eqref{eq:BEC}. We observe once again that Caffarelli and Lin obtained in \cite{CL} a result that is similar to our Theorem \ref{teo:main_result}, but only for the case when $U$ is a solution of \eqref{eq:BEC} having minimal energy.
\subsection{The class $\Seh(\Omega)$.}\label{subsec:class_S}

The second author of this paper, working in collaboration with Conti and Verzini, introduced in \cite{CTV2, CTV3} the following functional class:
\begin{multline*}
\Seh(\Omega)=\left\{ (u_1,\ldots, u_h)\in \left(H^1(\Omega)\right)^h:\ u_i\geq 0 \text{ in } \Omega,\ u_i\cdot u_j=0 \text{ if } i\neq j \text{ and } -\Delta u_i\leq f_i(u_i), \right.\\
\left.-\Delta( u_i-\sum_{j\neq i} u_j)\geq f_i(u_i)-\sum_{j\neq i} f_j(u_j) \text{ in }\Omega \text{ in the distributional sense}\right\}.
\end{multline*}
Here we make the following assumptions on the functions $f_i$:
\begin{itemize}
\item $f_i:\R^+\rightarrow \R$ Lipschitz continuous and $f_i(0)=0$;
\item there exists a constant $a<\lambda_1(\Omega)$ such that $|f_i(s)|\leq a s$ for every $x\in \Omega$, $s\geq \bar s>>1$.
\end{itemize}
This allows the use of the results of \cite{CTV3}. We stress that the conclusions of this subsection actually hold true for other different types of functions $f_i$, as for example the ones considered in \cite{CTV2}.

As observed in \cite{CTV2,CTV3,CTV4,CTV5,Caff_Aram_Lin,Zhitao,HHOT,HHOT2}, the class $\Seh(\Omega)$ is related to the asymptotic limits of reaction diffusion systems with a Lotka-Volterra--type competition term, as well to certain optimizations problems. We will recall some of these relations in the end of this subsection. The regularity results of Theorem \ref{teo:main_result} hold true for the elements of $\Seh(\Omega)$, as a byproduct of the following result.
\begin{teo}\label{teo:S_contained_in_G}$S(\Omega)\subseteq \Geh(\Omega)$.
\end{teo}
\begin{proof}
By the results proved in \cite{CTV3}, in order to obtain the desired conclusion the remaining thing to prove is that each $U\in \Geh(\Omega)$ satisfies property (G3). To prove it we follow the ideas contained in \cite[Theorem 15]{Caff_Aram_Lin}. Consider $\delta>0$ in such a way that each set $\{u_i>\delta\}$ is regular; moreover take $x_0\in \Omega$ and $r>0$. For simplicity we consider $F\equiv 0$. By using once again the Poh\u ozaev--type identity \eqref{eq:Pohoazaev} in each set $\{u_i>\delta\}\cap B_r(x_0)$ we obtain, by performing the same computations as in Lemma \ref{lemma:reflection_principle} and by passing to the limit superior as $\delta\rightarrow 0^+$,
$$
\frac{d}{dr}E(x_0,U,r)=\frac{2}{r^{N-2}}\int_{\partial B_r(x_0)}\left( \partial_\nu U\right)^2\, d\sigma + \frac{1}{r}\sum_i \limsup_{\delta\rightarrow 0^+}\int_{B_r(x_0)\cap \partial \{u_i>\delta\}} |\nabla u_i|^2\langle \nu,x-x_0\rangle \, d\sigma.
$$
Now fix $\varepsilon>0$ and define the set $S_\varepsilon=\{x\in\Omega\setminus \Gamma_U:\ \sum_i|\nabla u_i|\leq \varepsilon \}$. Since each component $u_i$ is harmonic in $\{u_i>0\}$, it is easy to prove the existence of a constant $C$ (independent of $\delta$) such that $\int_{B_r(x_0)\cap \partial\{u_i>\delta\} }|\nabla u_i|\, d\sigma\leq C.$ Thus
$$
\lim_{\varepsilon\rightarrow 0}\limsup_{\delta\rightarrow 0^+} \left| \int_{B_r(x_0)\cap \partial \{u_i>\delta\}\cap S_\varepsilon}|\nabla u_i|^2\langle \nu,x-x_0 \rangle \,d\sigma \right|  \leq \lim_{\varepsilon\rightarrow 0} C'\varepsilon=0.
$$
On the other hand at each point $x\in \Gamma_U\cap (\overline{\Omega\setminus S_\varepsilon})$ we have that $\nabla U(x)\neq 0$ (since $|\nabla U|$ is a continuous function, by \cite[Lemma 14]{Caff_Aram_Lin}). Thus in a small neighborhood of such $x$'s there exist exactly two components $u_i$ and $u_j$ (eventually different from point to point) and hence $\Delta (u_i-u_j)=0$, by taking into account the definition of $\Seh(\Omega)$. Therefore
$$
\sum_i \limsup_{\delta\rightarrow 0^+}\int_{B_r(x_0)\cap \partial\{u_i>\delta\}\cap(\Omega\setminus S_\varepsilon)}|\nabla u_i|^2 \langle \nu,x-x_0 \rangle  \,d\sigma =0.
$$
\end{proof}

As previously said, in recent literature it is proved that the solutions to several problems belong to $\Seh(\Omega)$ (in the following we recall two of them). Hence, we believe that Theorem \ref{teo:S_contained_in_G} is of great interest because it unifies several different points of view.

\subsubsection*{Lotka-Volterra competitive interactions}
 Consider the following Lotka-Volterra model for the competition between $h$ different species.
\begin{equation} \label{eq:lotka_volterra_system}
\left\{\begin{array}{l}
-\Delta u_i = f_i(u_i)-\beta u_i\sum_{j\neq i}u_j  \text{ in } \Omega,\\
u_i\geq 0 \text{ in } \Omega,\quad u_i=\varphi_i \text{ on } \partial \Omega.
  \end{array}\right.
\end{equation}
with $\Omega\subset \R^N$ a smooth bounded domain and $\varphi_i$ positive $W^{1,\infty}(\partial \Omega)$--functions with disjoint supports. The asymptotic study of its solutions (as $\beta\rightarrow +\infty$) has been the object of recent research, see for instance \cite{CTV5,Caff_Aram_Lin,Zhitao} and references therein. In \cite[Theorem 1]{CTV5} it is show that all the possible $H^1$--limits $U$ of a given sequence of solutions $\{U_\beta\}_{\beta>0}$ of \eqref{eq:lotka_volterra_system} (as $\beta\rightarrow +\infty$) belong to $\Seh(\Omega)$.

\subsubsection*{Regularity of interfaces in optimal partition problems}  Next we consider some optimal partition problems involving eigenvalues. For any integer $h\geq 0$, we define the set of $h$--partitions of $\Omega$ as
$$\Bger_h=\left\{(\omega_1,\ldots,\omega_h):\ \omega_i \text{ measurable },\ |\omega_i\cap\omega_j|=0 \text{ for }i\neq j \text{ and } \cup_i \omega_i\subseteq \Omega  \right\}. $$Consider the following optimization problems: for any positive real number $p\geq 1$,
\begin{equation}\label{eq:L_k_p}
\Lger_{h,p}:=\inf_{\Bger_h}\left(\frac{1}{h}\sum_{i=1}^h (\lambda_1(w_i))^p\right)^{1/p},
\end{equation}
and
\begin{equation}\label{eq:L_k}
\Lger_h:= \inf_{\Bger_h} \max_{i=1,\ldots,h}(\lambda_1(\omega_i)),
\end{equation}
where $\lambda_1(\omega)$ denotes the first eigenvalue of $-\Delta$ in $H^1_0(\omega)$ in a generalized sense (check \cite[Definition 3.1]{HHOT}). We refer to the papers \cite{CTV4,HHOT,CL_eigenvalues} for a more detailed description of these problems (in \cite{CTV4}, for instance, it is shown that \eqref{eq:L_k} is a limiting problem for \eqref{eq:L_k_p}, in the sense that $\lim_{p\rightarrow +\infty} \Lger_{h,p}=\Lger_h$). Our theory applies to opportune multiples of solutions of \eqref{eq:L_k_p} and \eqref{eq:L_k}. More precisely, in \cite[Lemma 2.1]{CTV4} it is shown that
\begin{itemize}
\item let $p\in[1,+\infty)$ and let $(\omega_1,\ldots,\omega_h)\in \Bger_h$ be any minimal partition associated with $\Lger_{h,p}$ and let $(\phi_i)_i$ be any set of positive eigenfunctions normalized in $L^2$ corresponding to $(\lambda_1(\omega_i))_i$. Then there exist $a_i>0$ such that the functions $u_i=a_i \phi_i$ verify in $\Omega$, for every $i=1,\ldots, h$, the variational inequalities $-\Delta u_i\leq \lambda_1(\omega_i) u_i$ and $-\Delta (u_i-\sum_{j\neq i} u_j)\geq \lambda_1(\omega_i) u_i-\sum_{j\neq i}\lambda_1(\omega_i) u_j$ in the distributional sense;
\end{itemize}
and in \cite[Theorem 3.4]{HHOT}:
\begin{itemize}
\item let $(\tilde \omega_1,\ldots,\tilde \omega_h)\in \Bger_h$ be any minimal partition associated with $\Lger_h$ and let $(\tilde \phi_i)_i$ be any set of positive eigenfunctions normalized in $L^2$ corresponding to $(\lambda_1(\tilde \omega_i))_i$. Then there exist $a_i\geq 0$, not all vanishing, such that the functions $\tilde u_i=a_i \tilde \phi_i$ verify in $\Omega$, for every $i=1,\ldots, h$, the variational inequalities $-\Delta \tilde u_i\leq \Lger_h \tilde u_i$ and $-\Delta (\tilde u_i-\sum_{j\neq i}\tilde u_j)\geq \Lger_h (\tilde u_i-\sum_{j\neq i}\tilde u_j)$ in the distributional sense.
\end{itemize}
In particular the functions $\tilde U=(\tilde u_1,\ldots,\tilde u_h)$ and $U=(u_1,\ldots,u_h)$ belong to $\Seh(\Omega)$.

We refer to the book \cite{BB} for other interesting optimization problems. It is our belief that the solutions to some of these problems should belong to the class $\Geh(\Omega)$.

\noindent {\bf Acknowledgments.}
The first author was supported
by FCT, grant SFRH/BD/28964/2006 and Financiamento
Base 2008 - ISFL/1/209.


\noindent \verb"htavares@ptmat.fc.ul.pt"\\
University of Lisbon, CMAF, Faculty of Science, Av. Prof. Gama Pinto
2, 1649-003 Lisboa, Portugal

\noindent \verb"susanna.terracini@unimib.it"\\
Dipartimento di Matematica e Applicazioni, Universit\`a degli Studi
di Milano-Bicocca, via Bicocca degli Arcimboldi 8, 20126 Milano,
Italy

\end{document}